\documentclass
[
    DIV=10,
    abstracton
]
{scrartcl}

\usepackage{amsfonts,amssymb,amsmath,amsthm}
\usepackage{authblk}
\usepackage{bm}
\usepackage{bbm}
\usepackage{dsfont}
\usepackage{tabularx}
\usepackage[pdftex]{graphicx}
\usepackage{xcolor}
\usepackage{epstopdf}
\usepackage{enumerate}  
\usepackage{tikz} 
\usetikzlibrary{decorations.markings} 
\usetikzlibrary{arrows,decorations.pathmorphing,backgrounds,fit,positioning,shapes.symbols,chains}
\usepackage{pgfplots}
\usepgfplotslibrary{fillbetween}

\usepackage[colorlinks,
			pdffitwindow=false,
      plainpages=false,
      pdfpagelabels=true,
     	pdfpagemode=UseOutlines,
      pdfpagelayout=SinglePage,
			bookmarks=false,
			colorlinks=true,
      hyperfootnotes=false,
			linkcolor=blue,
			citecolor=green!50!black]{hyperref}

\usepackage{paralist} 
\usepackage{algorithm}
\usepackage[noend]{algpseudocode}

\graphicspath{{./graphics/}}

\newtheoremstyle{custom}{3pt}{3pt}{}{}{\bfseries}{:}{.5em}{}
\theoremstyle{custom}
\newtheorem{example}    {Example}

\newtheorem{theorem}    [example]{Theorem}

\newtheorem{remark}[example]{Remark}
\numberwithin{example}{section}

\newcommand{\A}{\op{A}}

\newcommand{\R}{\mathbb{R}}
\newcommand{\C}{\mathbb{C}}

\newcommand{\set}[1]{\mathbb{#1}}		
\newcommand{\X}{\set{X}}

\newcommand{\op}[1]{\mathcal{#1}}	
\def \prob {\mathsf{P}} 
\def \expect {\mathsf{E}} 
\newcommand{\innerprod}[2]{\langle #1, #2 \rangle}
\newcommand{\T}{\op{T}}

\renewcommand{\P}{\op{P}}
\newcommand{\K}{\op{K}}
\newcommand{\Q}{\op{Q}}
\newcommand{\one}{\mathds{1}}
\newcommand{\dat}[1]{\bm{#1}}
\newcommand{\rhoref}{\mu_{\text{ref}}}

\DeclareMathOperator*\argmin{\arg\,\min}

\title{Optimal data-driven estimation of generalized Markov state models for non-equilibrium dynamics}

\author[1]{P\'eter Koltai}
\author[1]{Hao Wu}
\author[1]{Frank No\'e}
\author[1,2]{Christof Sch\"utte}

\affil[1]{\normalsize Department of Mathematics and Computer Science, Freie Universit\"at Berlin, Germany}
\affil[2]{\normalsize Zuse Institute Berlin, Germany}

\date{}

\begin{document}

\maketitle

\begin{abstract}
There are multiple ways in which a stochastic system can be out of statistical equilibrium. It might be subject to time-varying forcing; or be in a transient phase on its way towards equilibrium; it might even be in equilibrium without us noticing it, due to insufficient observations; and it even might be a system failing to admit an equilibrium distribution at all.
We review some of the approaches that model the effective statistical behavior of equilibrium and non-equilibrium dynamical systems, and show that both cases can be considered under the unified framework of optimal low-rank approximation of so-called transfer operators. Particular attention is given to the connection between these methods, Markov State models, and the concept of metastability, further to the estimation of such reduced order models from finite simulation data. We illustrate our considerations by numerical examples.
\end{abstract}


\section{Introduction}


The term ``equilibrium'' is not used uniformly throughout the literature. So, to start off with, an equilibrium process in this paper means a reversible one in the stochastic sense, see Table~\ref{tab:nomen}.

Metastable molecular systems under non-equilibrium conditions caused by external fields have attracted increasing
interest recently. For instance, new experimental techniques like atomic force microscopy or simulation studies regarding the potential effects of electromagnetic
radiation on the human body tissue  have been extensively
investigated in the literature.  Specifically adapted molecular dynamics (MD) simulations have proved particularly useful for understanding the response of biomolecular conformations to external fields. 
Despite this significance, reliable tools for the quantitative description of non-equilibrium phenomena like
the conformational dynamics of a molecular system under external forcing are still lacking. 

For MD simulations in equilibrium such specific and reliable tools have been developed: Markov State Models (MSMs) allow for an accurate description of the transitions between the main conformations of the molecular system under investigation. MSMs for equilibrium MD have been well developed over the
past decade in theory~\cite{SchSa13,prinz2011markov}, applications (see the recent book~\cite{A19-1} for an overview), and software implementations~\cite{A19-49, MSMBuilder17}. They now form a set of standard tools.
The principal idea of equilibrium MSMs is to approximate the MD system (in continuous state or phase space) by
a reduced Markovian dynamics over a finite number of (macro-)states (i.e., in discrete state space). These (macro-)states represent the dominant metastable sets of the system, i.e., sets in which typical MD trajectories stay substantially longer than the system needs for a transition to another such set~\cite{SchSa13,schuette2011markov}. In equilibrium MD, these metastable sets are the main conformations of the molecular system under consideration which, often enough, are given by the main wells in its energy landscape. 
It has been shown that for many (bio)molecular systems  the Markovian dynamics given by an MSM allows very close approximation of the longest relaxation processes of the underlying molecular system under equilibrium conditions~\cite{sarich2010approximation,Eigenvalues}. 

However, in the non-equilibrium setting the above tools are not guaranteed to continue working. Note that there are different possibilities to deviate from the ``equilibrium'' situation, and this makes the term ``non-equilibrium'' ambiguous. To avoid confusion, we consider one of the following cases when referring to the non-equilibrium setting (again, see Table~\ref{tab:nomen} on terminology).
\begin{compactenum}[(i)]
\item
\emph{Time-inhomogeneous dynamics}, e.g., the system feels a time-dependent external force, for instance due to an electromagnetic field or force probing.
\item
\emph{Time-homogeneous non-reversible dynamics}, i.e., where the governing laws of the system do not change in time, but the system does not obey detailed balance, and, additionally we might want to consider the system in a non-stationary regime.
\item
\emph{Reversible dynamics but non-stationary data}, i.e., the system possesses a stationary distribution with respect to which it is in detailed balance, but the empirical distribution of the available data did not converge to this stationary distribution.
\end{compactenum}
Even though we consider genuinely stochastic systems here, the algorithm of section~\ref{sec:data based approx} can be used for deterministic systems as well---and indeed it is, see Remark~\ref{rem:datameth} and references therein.

Note that with regard to the considered dynamics (i)--(iii) represent cases with decreasing generality.
For (i), time-dependent external fields act on the system, such that the energy landscape depends on time, i.e., the main wells of the energy landscape can move in time. That is, there may no longer be time-independent metastable sets in which the dynamics stays for long periods of time before exiting. Instead, the potentially metastable sets will move in state space. Generally, moving ``metastable'' sets cannot be considered metastable anymore. However, the so-called \emph{coherent sets}, which have been studied for non-autonomous flow fields in fluid dynamics~\cite{FrSaMo10,Froyland2013}, permit to get a meaning to the concept of metastability~\cite{KoCiSch16}. 
For (iii), the full theory of equilibrium Markov state modeling is at one's disposal, but one needs to estimate certain required quantities from non-equilibrium data~\cite{WuEtAl17}. Case (ii) seems the most elusive, due to the fact that on the one hand it could be handled by the time-inhomogeneous approach, but on the other hand it is a time-homogeneous system and some structural properties could be carried over from the reversible equilibrium case that are out of reach for a time-inhomogeneous analysis. For instance, if the dynamics shows cyclic behavior, it admits structures that are well captured by tools from the analysis of time-homogeneous dynamical systems (e.g., Floquet theory and Poincar\'e sections~\cite{MSM_periodic, FrKo17}), and a more general view as in (i) might miss them; however, cyclic behavior is not present in reversible systems, such that the tools from (iii) are doomed to failure in this respect.
In order to avoid confusion, however, it should be emphasized that the three cases distinguished above do not suffice to clarify the discussion about the definition of equilibrium or non-equilibrium, e.g., see the literature on non-equilibrium steady state (NESS) systems~\cite{seifert2010fluctuation,LeLaPa17}.

Apart from MSMs the literature on kinetic lumping schemes offers several other techniques for finding a coarse-grained descriptions of systems~\cite{Nystrom,Network,KnSp15}. These techniques are, however, not built on the intuition of metastable behavior in state space. What we consider here can be seen in connection to optimal prediction in the sense of the Mori--Zwanzig formalism~\cite{mori1965transport,Zwa73,ChHaKu00,ChHaKu02}, but we will try to choose the observables of the system such that projecting the dynamics on these keeps certain properties intact without including memory terms.

The aim of this article is to review and unify some of the theoretical and also data-driven algorithmic approaches that attempt to model the effective statistical behavior of non-equilibrium systems. To this end, a MSM, or, more precisely, a \emph{generalized MSM} is sought, i.e., a possibly small matrix~$T_k$ that carries the properties of the actual system that are of physical relevance. In the equilibrium case, for example, this includes the slowest timescales on which the system relaxes towards equilibrium (section~\ref{sec:dynamics and functions}). The difference of generalized to standard MSMs is that we do not strictly require the former to be interpretable in terms of transition probabilities between some regions of the state space (section~\ref{sec:MSMeq}), however usually there is a strong connection between the matrix entries and metastable sets. We will, however, focus on a slightly different characteristic of the approximate model, namely its ``propagation error'', that will allow for a straightforward generalization from equilibrium (reversible) to all our non-equilibrium cases (section~\ref{sec:MSMneq}); and even retain the physical intuition behind true MSMs through the concept of coherent sets. We will show in section~\ref{sec:data based approx} how these considerations can be carried over to the case when only a finite amount of simulation data is available. The above non-equilibrium cases (ii)--(iii) can be then given as specific instances of the construction (section~\ref{sec:nonstat timehomog}). The theory is illustrated with examples throughout the text.

We note in advance that in course of the (generalized) Markov state modeling we will consider different instances of approximations to a certain linear operator~$\T:\set{S}\to\set{S}$ mapping some space to itself (and sometimes to a different one). On the one hand, there will be a projected operator~$\T_k:\set{S}\to\set{S}$, where~$\T_k = \Q\T\Q$ with a projection~$\Q:\set{S}\to\set{V}$ having a $k$-dimensional range~$\set{V} \subset \set{S}$. On the other hand, we will consider the restriction of the projected operator~$\T_k$ to this $k$-dimensional subspace, i.e.,~$\T_k:\set{V} \to \set{V}$, also called $\set{V}$-restriction of~$\T_k$, which has a $k\times k$ matrix representation (with respect to some chosen basis of~$\set{V}$) that we will denote by~$T_k$.
\begin{table}[h]
\renewcommand{\arraystretch}{1.25}
\begin{tabularx}{\textwidth}{l X}
\hline
\multicolumn{2}{l}{A stochastic (Markov) process is called...} \\
\emph{time-homogeneous} & if the transition probabilities from time $s$ to time $t$ depend only on $t-s$ (in analogy to the evolution of an autonomous ODE). \\
\emph{stationary} & if the distribution of the process does not change in time (such a distribution is also called invariant, cf.~\eqref{eq:statdens}). \\
\emph{reversible} & if it is stationary and the detailed balance condition~\eqref{eq:detbal} holds (reversibility means that time series are statistically indistinguishable in forward and backward time).\\
\hline
\end{tabularx}
\caption{Nomenclature used here for stochastic processes.}
\label{tab:nomen}
\end{table}

\section{Studying dynamics with functions}
\label{sec:dynamics and functions}

\subsection{Transfer operators}

In what follows, $ \prob [\,\cdot \mid \mathfrak{E}] $ and $ \expect  [\, \cdot \mid \mathfrak{E}] $ denote probability and expectation conditioned on the event $ \mathfrak{E} $. Furthermore, $ \{x_t\}_{t \ge 0} $ is a stochastic process defined on a state space $ \X \subset \R^d $. For instance, we can think of $x_t$ being the solution of the stochastic differential equation
\begin{equation}
    \mathrm{d} x_t = -\nabla W(x_t)\,\mathrm{d}t + \sqrt{2\beta^{-1}} \, \mathrm{d}w_t\,,
    \label{eq:overdampedLangevin}
\end{equation}
describing diffusion in the potential energy landscape given by~$W$. Here,~$\beta$ is the non-dimensionalized inverse temperature, and $w_t$ is a standard Wiener process (Brownian motion).
The \emph{transition density function} $ p^t \colon \X \times \X \to \R_{\ge 0} $ of a time-homogeneous stochastic process $ \{ x_t \}_{t \ge 0} $ is defined by
\begin{equation*}
    \prob [ x_t \in \set{A} \mid x_0 = x] = \int_{\set{A}} p^t(x,y) \, \mathrm{d}y\,,\qquad \set{A}\subseteq \X\,.
\end{equation*}
That is, $ p^t(x,y) $ is the conditional probability density of $ x_t = y $ given that $ x_0 = x $. We also write~$x_t \sim p^t(x_0,\cdot)$ to indicate that $x_t$ has density~$p^t(x_0,\cdot)$.

With the aid of the transition density function, we will now define transfer operators, i.e., the action of the process on functions of the state. Note, however, that the transition density is in general not known explicitly, and thus we will need data-based approximations to estimate it.
We assume that there is a \emph{unique} stationary density~$ \mu $, such that $ \{ x_t\}_{t \ge 0} $ is stationary with respect to~$\mu$; that is, it satisfies~$x_0\sim\mu$ and
\begin{equation} \label{eq:statdens}
\mu(x) = \int_\X \mu(y) \, p^t(y,x)  \mathrm{d}y \quad\text{for all }t\ge 0.
\end{equation}

Let now $f$ be a probability density over~$\X$, $u = f/\mu$ a probability density with respect to $\mu$ (meaning that $u\mu$ is to be interpreted as a physical density), and $g$ a scalar function of the state (an ``observable''). We define the following \emph{transfer operators},  for a given lag time $ \tau $:
\begin{compactenum}[(a)]
\item The \emph{Perron--Frobenius operator} (also called propagator),
\begin{equation*}
    \mathcal{P}^\tau f(x) = \int_{\X} f(y)\, p^\tau(y,x) \, \mathrm{d}y
\end{equation*}
evolves probability distributions.
\item The \emph{Perron--Frobenius operator with respect to the equilibrium density} (also called transfer operator, simply),
\begin{equation*}
    \T^\tau u(x) = \frac{1}{\mu(x)} \int_{\X} u(y)\mu(y)\, p^\tau(y,x)\,\mathrm{d}y\,.
\end{equation*}
evolves densities with respect to~$\mu$.
\item The \emph{Koopman operator}
\begin{equation} \label{eq:Koopman operator}
    \K^\tau g(x) = \int_{\X} p^\tau(x,y) \, g(y) \,\mathrm{d}y
                       = \expect [ f(x_\tau) \mid x_0 = x]
\end{equation}
evolves observables.
\end{compactenum}
All our transfer operators are well-defined non-expanding operators on the following Hilbert spaces:\footnote{We denote by~$L^q = L^q(\X)$ the space (equivalence class) of $q$-integrable functions with respect to the Lebesgue measure. $L^q_{\nu}$ denotes the same space of function, now integrable with respect to the weight function~$\nu$.},~$\mathcal{P}^\tau:L^2_{1/\mu} \to L^2_{1/\mu}$, $\T^\tau: L^2_{\mu} \to L^2_{\mu}$, and~$\K^\tau: L^2_{\mu} \to L^2_{\mu}$~\cite{BaRo95,SchCa92,KNKWKSN17}.
The equilibrium density $ \mu $ satisfies $ \mathcal{P}^\tau \mu=\mu$, that is, $ \mu $ is an eigenfunction of $ \mathcal{P}^\tau $ with associated eigenvalue $ \lambda_0 = 1 $. The definition of $ \T^\tau $ relies on $\mu$, we have
\begin{equation} \label{eq:PT_identity}
\mu\, \T^\tau u = \mathcal{P}^\tau (u \mu)\,,
\end{equation}
thus $\P^\tau \mu = \mu$ translates into $\T^\tau \one = \one$, where $\one = \one_{\X}$ is the constant one function on~$\X$.

\subsection{Reversible equilibrium dynamics and spectral decomposition}

An important structural property of many systems used to model molecular dynamics is \emph{reversibility}. Reversibility means that the process is statistically indistinguishable from its time-reversed counterpart, and it can be described by the detailed balance condition
\begin{equation} \label{eq:detbal}
    \mu(x) \, p^t(x,y) = \mu(y) \, p^t(y,x)\qquad \forall x, y \in \X,\ t\ge 0\,.
\end{equation}

The process generated by~\eqref{eq:overdampedLangevin} is reversible and ergodic, i.e., it admits a unique positive equilibrium density, given by~$\mu(x) \propto \exp(-\beta W(x))$, under mild growth conditions on the potential~$W$~\cite{MaSt02,MaStHi02}. Note that the subsequent considerations hold for all stochastic processes that satisfy reversibility and ergodicity with respect to a unique positive invariant density and are \emph{not} limited to the class of dynamical systems given by~\eqref{eq:overdampedLangevin}. See \cite{SchSa13} for a discussion of a variety of stochastic dynamical systems that have been considered in this context. Furthermore, if~$p^t(\cdot,\cdot)$ is a continuous function in both its arguments for~$t>0$, then all the transfer operators above are compact, which we also assume from now on. This implies that they have a discrete eigen- and singular spectrum (the latter meaning it has a discrete set of singular values). For instance, the process generated by~\eqref{eq:overdampedLangevin} has has continuous transition density function under mild growth and regularity assumptions on the potential~$W$.

As a result of the detailed balance condition, the Koopman operator $ \K^\tau $ and the Perron--Frobenius operator with respect to the equilibrium density $ \T^\tau $ become identical and we obtain
\begin{equation} \label{eq:selfadj}
    \innerprod{\mathcal{P}^\tau f}{g}_{1/\mu} = \innerprod{f}{\mathcal{P}^\tau g}_{1/\mu}
    \quad \text{and} \quad
    \innerprod{\T^\tau f}{g}_\mu     = \innerprod{f}{\T^\tau g}_\mu\,,
\end{equation}
i.e., all the transfer operators become self-adjoint on the respective Hilbert spaces from above.
Here,~$\innerprod{\cdot}{\cdot}_{\nu}$ denotes the natural scalar products on the weighted space~$L^2_{\nu}$, i.e.,~$\innerprod{f}{g}_{\nu} = \int_{\X}f(x)g(x)\nu(x)\,\mathrm{d}x$; the associated norm is denoted by~$\| \cdot \|_{\nu}$.
Due to the self-adjointness, the eigenvalues $ \lambda_i^\tau $ of $ \mathcal{P}^\tau $ and $ \T^\tau $ are real-valued and the eigenfunctions form an orthogonal basis with respect to $ \innerprod{\cdot}{\cdot}_{1/\mu} $ and $ \innerprod{\cdot}{\cdot}_\mu $, respectively.


Ergodicity implies that the dominant eigenvalue $ \lambda_1 $ is the only eigenvalue with absolute value $ 1 $ and we can thus order the eigenvalues so that
\begin{equation*}
    1 = \lambda_1^\tau > \lambda_2^t \ge \lambda_3^t \ge \dots.
\end{equation*}
The eigenfunction of~$\T^\tau$ corresponding to $ \lambda_1 = 1 $ is the constant function $ \phi_1 = \mathds{1}_{\X} $. Let $ \phi_i $ be the normalized eigenfunctions of $ \T^\tau $, i.e.~$ \innerprod{\phi_i}{\phi_j}_\mu = \delta_{ij} $, where~$\delta_{ij}$ denotes the Kronecker-delta.
Then any function~$ f \in L_{\mu}^2 $ can be written in terms of the eigenfunctions as $ f = \sum_{i=1}^\infty \innerprod{f}{\phi_{i}}_\mu \, \phi_i $. Applying $ \T^\tau $ thus results in
\begin{equation*}
    \T^\tau f = \sum_{i=1}^\infty \lambda_i^\tau \, \innerprod{f}{\phi_i}_\mu \, \phi_i.
\end{equation*}
For more details, we refer to~\cite{KNKWKSN17} and references therein.

For some $ k \in \mathbb{N} $, we call the $ k $ dominant eigenvalues $ \lambda_1^\tau, \dots, \lambda_k^\tau $ of $ \T^\tau $ the \emph{dominant spectrum} of $ \T^\tau $, i.e.,
\begin{equation*}
    \lambda_\text{dom}(\T^\tau) = \{ \lambda_1^\tau, \dots, \lambda_k^\tau \}.
\end{equation*}
Usually, $ k $ is chosen in such a way that there is a \emph{spectral gap} after $ \lambda_k^\tau $, i.e.~$ 1-\lambda_k^\tau \ll \lambda_k^\tau - \lambda_{k+1}^\tau $. The \emph{(implied) time scales} on which the associated dominant eigenfunctions decay are given by
\begin{equation} \label{eq:implied time scales}
    t_i = -\tau/\log(\lambda_i^\tau).
\end{equation}
If $ \{\T^t\}_{t\ge 0} $ is a semigroup of operators (which is the case for every time-homogeneous process, as, e.g., the transfer operator associated with~\eqref{eq:overdampedLangevin}), then there are $ \kappa_i \le 0 $ with $ \lambda_i^\tau = \exp(\kappa_i \tau) $ such that $t_i = -\kappa_i^{-1} $ holds. Assuming there is a spectral gap, the dominant time scales satisfy $ \infty = t_1 > \ldots \ge t_k \gg t_{k+1} $. These are the time scales of the \emph{slow} dynamical processes, also called \emph{rare events}, which are of primary interest in applications. The other, \emph{fast} processes are regarded as fluctuations around the relative equilibria (or \emph{metastable states}) between which the relevant slow processes travel.

\section{Markov state models for reversible systems in equilibrium}
\label{sec:MSMeq}

In the following, we will fix a lag time $\tau>0$, and drop the superscript $ \tau $ from the transfer operators for clarity of notation.

\subsection{Preliminaries on equilibrium Markov state models}

Generally, in the equilibrium case, a generalized MSM (GMSM) is any matrix $T_k \in \R^{n_k\times n_k}$, $n_k\ge k$, that approximates the~$k$ dominant time scales of~$\T$, i.e., its dominant eigenvalues;
\begin{equation}
\lambda_{\text{dom}}(T_k) \approx \lambda_{\text{dom}}(\T)\,.
\end{equation}
It is natural to ask for some structural properties of $\T$ to be reproduced by $T_k$, such as:
\begin{compactitem}
\item $\T$ is a positive operator $\longleftrightarrow$ all entries of $T_k$ are non-negative;
\item $\T$ is probability-preserving $\longleftrightarrow$ each column sum of $T_k$ is~1.
\end{compactitem}
These two properties together make $T_k$ to a stochastic matrix, and in this case~$T_k$ is usually called a MSM. We shall use the term Generalized MSM for a matrix $T_k$ that violates these requirements but still approximates the dominant spectral components of the underlying operator. Another structural property that one would usually ask for is to have apart from the time scales/eigenvalues also some approximation of the associated eigenvectors of~$\T$, as these are the dynamic observables related to the slow dynamics. This is incorporated in the general approach, which we discuss next.

The question is now \emph{how} to obtain a GMSM $T_k$ for a given~$\T$? To connect these objects, a natural and popular approach is to obtain the reduced model~$T_k$ via \emph{projection}. To this end, let~$\Q:L^2_{\mu}\to\set{V}\subset L^2_{\mu}$ be a projection onto a~$n_k$-dimensional subspace~$\set{V}$. The GMSM is then defined by the projected transfer operator
\begin{equation} \label{eq:MSMeq}
\T_k = \Q\T\Q\,;
\end{equation}
and $T_k$ can now be taken as the matrix representation of the $\set{V}$-restriction of the projected operator~$\T_k:\set{V}\to\set{V}$ with respect to a chosen basis of~$\set{V}$.

Is there a ``best'' choice for the projection? If we also ask for perfect approximation of the time scales, i.e.,~$\lambda_{\text{dom}}(T_k) = \lambda_{\text{dom}}(\T)$, the requirement of parsimony---such that the model size is minimal, i.e.,~$n_k=k$---leaves us with a unique choice for~$\set{V}$, namely the space spanned by the dominant (normalized) eigenfunctions~$\phi_i$ of~$\T$,~$i=1,\ldots, k$.
This follows from the so-called \emph{variational principle} (or Rayleigh--Ritz method)~\cite{NoNu13,NuEtAl14}. In fact, it makes a stronger claim: every projection to a $k$-dimensional space~$\set{V}'$ yields a GMSM~$\T_k': \set{V}' \to \set{V}'$ which underestimates the dominant time scales, i.e.,~$\lambda_i(\T_k') \le \lambda_i(\T)$,~$i=1,\ldots,k$; and equality holds only for the projections on the eigenspaces.

Note that the discussion about the time scales (equivalently, the eigenvalues) involves only the \emph{range} of the projection, the space~$\set{V}$.
However, there are multiple ways to project on the space~$\set{V}$. It turns out, that the $\mu$-orthogonal projection given by
\begin{equation} \label{eq:projeq}
\Q f = \sum_{i=1}^k \innerprod{\phi_i}{f}_{\mu}\, \phi_i
\end{equation}
is superior to all of them, if we consider a \emph{stronger} condition than simply reproducing the dominant time scales. This condition is the requirement of minimal propagation error, and it will be central to our generalization of GMSMs for non-equilibrium, or even time-inhomogeneous systems.
Let us define the best $k$-dimensional approximation~$\T_k$ to~$\T$, i.e., the best projection~$\Q$, as the rank-$k$ operator satisfying
\begin{equation} \label{eq:normestim}
\|\T-\T_k\| \le \| \T-\T_k'\| \,,
\end{equation}
where~$\|\cdot\|$ denotes the induced operator norm\footnote{The induced norm of an operator~$\A : \set{X} \to \set{Y}$ is defined by~$\| \A \| = \max_{\|f\|_{\set{X}} = 1}\|\A f\|_{\set{Y}}$, where~$\|\cdot\|_{\set{X}}$ and~$\|\cdot\|_{\set{Y}}$ are the norms on the spaces~$\set{X}$ and~$\set{Y}$, respectively.} for operators mapping~$L^2_{\mu}$ to itself.
Equivalently, this can be viewed as a result stating that~$\T_k$ is the $k$-dimensional approximation of~$\T$ yielding the smallest (worst-case) error in density propagation:
\begin{equation} \label{eq:densprop}
\T_k = \argmin_{\substack{\T_k' = \Q' \T  \Q' \\ \text{rank}\,\Q'=k}} \max_{\|f\|_\mu = 1}  \| \T  f-\T_k' f \|_\mu \,,
\end{equation}
where $x_* = \argmin_x h(x)$ means that~$x_*$ minimizes the function~$h$, possibly subject to constraints that are listed under $\argmin$.

To summarize, the best GMSM~\eqref{eq:MSMeq} in terms of~\eqref{eq:normestim} (or, equivalently,~\eqref{eq:densprop}) is given by the projection~\eqref{eq:projeq}. This follows from the self-adjointness of~$\T$ and the Eckard--Young theorem; details can be found in~\cite{WuNo17} and in Appendix~\ref{app:E-Ythm}.
Caution is needed however, when interpreting~$\T_k f$ as the propagation of a given probability density~$f$. The projection to the dominant eigenspace in general does not respect positivity (i.e., $f\ge 0 \nRightarrow \T_kf\ge 0$), thus~$\T_k f$ loses its probabilistic meaning. This is the price to pay for the perfectly reproduced dominant time scales. We can retain a physical interpretation of a MSM if we accept that the dominant time scales will be slightly off, as we discuss in the next section.

\subsection{Metastable sets}
\label{ssec:metsets}

There is theoretical evidence~\cite{SchSa13} that the more pronounced the metastable behavior of system is (in the sense that the size of the time scale gap~$ t_1 \ge \ldots \ge t_k \gg t_{k+1} $ is large), the more constant the dominant eigenfunctions $\phi_i$ are on the metastable sets~$\set{M}_1,\ldots,\set{M}_d$, given the lag time with respect to which the transfer operator~$\T  = \T^\tau$ is taken satisfies~$\tau \gg t_{k+1}$. Assuming such a situation, the eigenfunctions of~$\T$ can approximately be combined from the characteristic functions over the metastable sets, i.e., with the abbreviation $\one_i := \one_{\set{M}_i}$ it holds that
\begin{equation} \label{eq:EVlincomb}
\phi_i \approx \sum_{j=1}^k c_{ij}\one_j =: \widehat{\phi}_i \,,
\end{equation}
where the~$c_{ij}$ are components of the linear combination, such that the $\widehat{\phi}_i$ are orthonormal. Using the ``approximate eigenfunctions''~$\widehat{\phi}_i$ defined in~\eqref{eq:EVlincomb}, the modified projection
\begin{equation} \label{eq:met_proj}
\widehat{\Q}f = \sum_{i=1}^k \innerprod{\widehat{\phi}_i}{f}_{\mu} \widehat{\phi}_i
\end{equation}
defines a new MSM~$\widehat{\T}_k := \widehat{\Q} \T  \widehat{\Q}$. Since $\set{V} = \text{span}\{\phi_i\} \approx \text{span}\{ \widehat{\phi}_i\} = \widehat{\set{V}}$, also~$\widehat{\Q} \approx \Q$, and thus we have~$\widehat{\T}_k \approx \T_k$. This implies~\cite[Lemma~3.5]{BitEtAl17} that also their dominant eigenvalues, hence time scales are close. Further, we have that in the basis $\{ \one_i / \innerprod{\one_i}{\one_i}_\mu \}_{i=1}^k$ the matrix representation~$\widehat{T}_k$ of the $\widehat{\set{V}}$-restriction of the operator~$\widehat{\T}_k$ has the entries
\begin{equation} \label{eq:met_transrates}
\begin{aligned}
\widehat{T}_{k,ij} &= \frac{\innerprod{\one_i}{\T\one_j}_{\mu}}{\innerprod{\one_j}{\one_j}_{\mu}} \\
&=  \int_{\set{M}_i} \T\left(\frac{\one_j}{\innerprod{\one_j}{\one_j}_{\mu}}\right)\mu(x)dx \\
&= \frac{1}{\prob_\mu [x_0\in\set{M}_j]} \int_{\set{M}_i}  \int_{\set{M}_j} \mu(x) p^t(x,y)\,dx\,dy \\
&= \prob_{\mu} [x_t \in\set{M}_i\,\big\vert\, x_0 \in\set{M}_j]\,,
\end{aligned}
\end{equation}
where~$\prob_\mu[\,\cdot\,\vert\, x\in\set{M}]$ denotes the probability measure that arises if~$x\in\set{M}$ has distribution~$\mu$ (restricted to~$\set{M}$). That is,~$\widehat{T}_k$ has the transition probabilities between the metastable sets as entries, giving a direct physical interpretation of the MSM.  Note, however, that for this approximation to reproduce the dominant time scales well, i.e., to have~$t_i \approx  \widehat{t}_i$, $i=1,\ldots,k$, we need a strong separation of time scales in the sense that~$t_k \gg t_{k+1}$ has to hold, and the lag time~$ \tau $ needs to be chosen sufficiently large~\cite{sarich2010approximation}.

\subsection{Example: stationary diffusion in double-well potential}
\label{ssec:stationaryDW}

Let us consider the diffusion~\eqref{eq:overdampedLangevin} in the potential landscape~$W(x) = (x^2-1)^2$ with~$\beta = 5$; cf.~Figure~\ref{fig:dw_stationary} (left). With the lag time~$\tau=10$ we approximate the Perron--Frobenius operator~$\P = \P^t$ and compute its eigenvector~$\mu$ at the eigenvalue~$\lambda_1=1$. Then, we compute the transfer operator~$\T = \T^\tau$ with respect to the stationary distribution~$\mu$, and its dominant eigenvalues $\lambda_2,\lambda_3,\ldots$ and corresponding eigenvectors~$\phi_2,\phi_3,\ldots$ (Figure~\ref{fig:dw_stationary}, right). While $\lambda_2 = 0.888$, we have~$|\lambda_3| < 10^{-12}$, hence we have a clear time scale separation,~$t_2 = 84.1 \gg 0.35 = t_3$, cf.~\eqref{eq:implied time scales}.
\begin{figure}[htb]
\centering
\includegraphics[width = 0.394\textwidth]{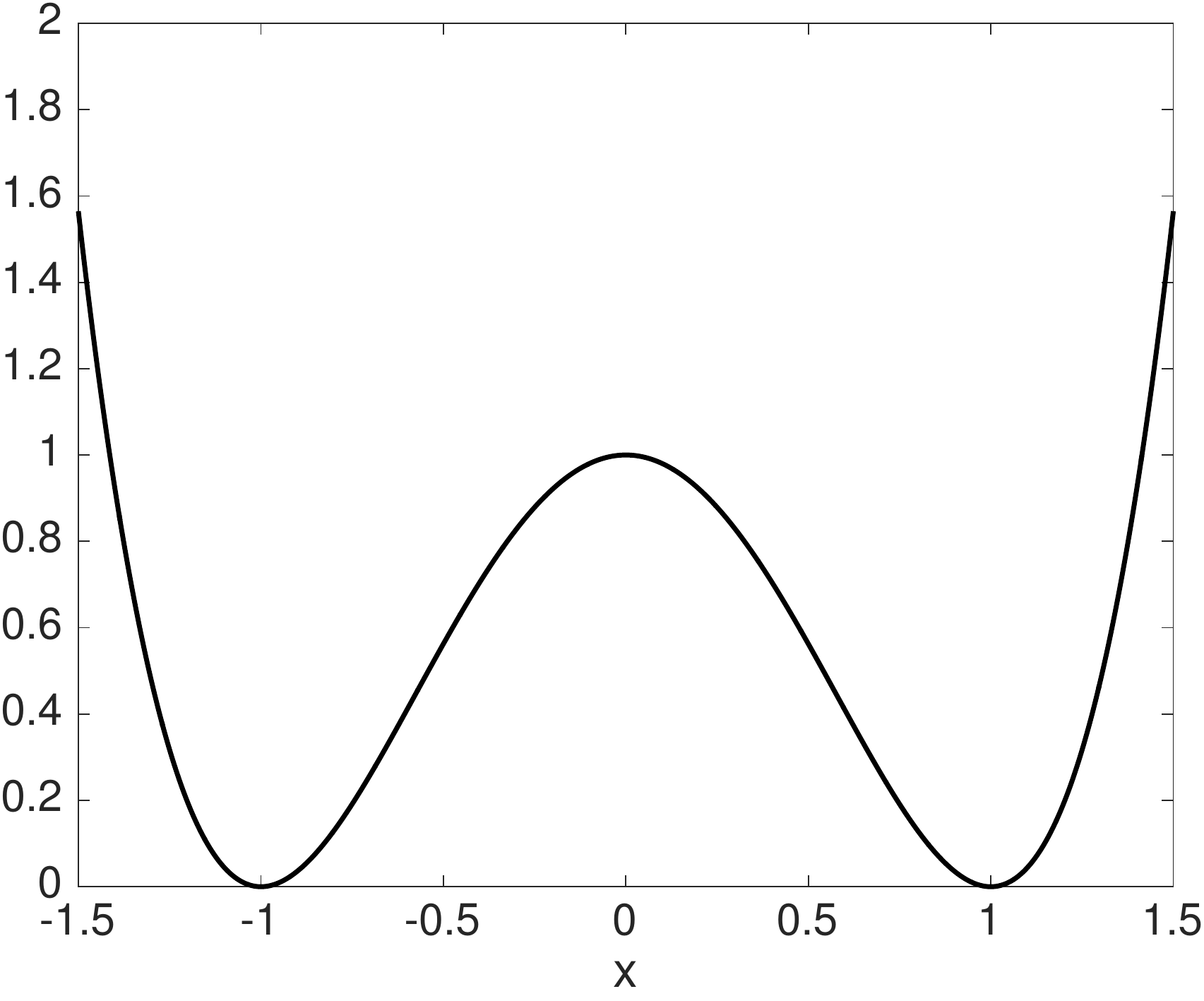}
\hspace*{1cm}
\includegraphics[width = 0.4\textwidth]{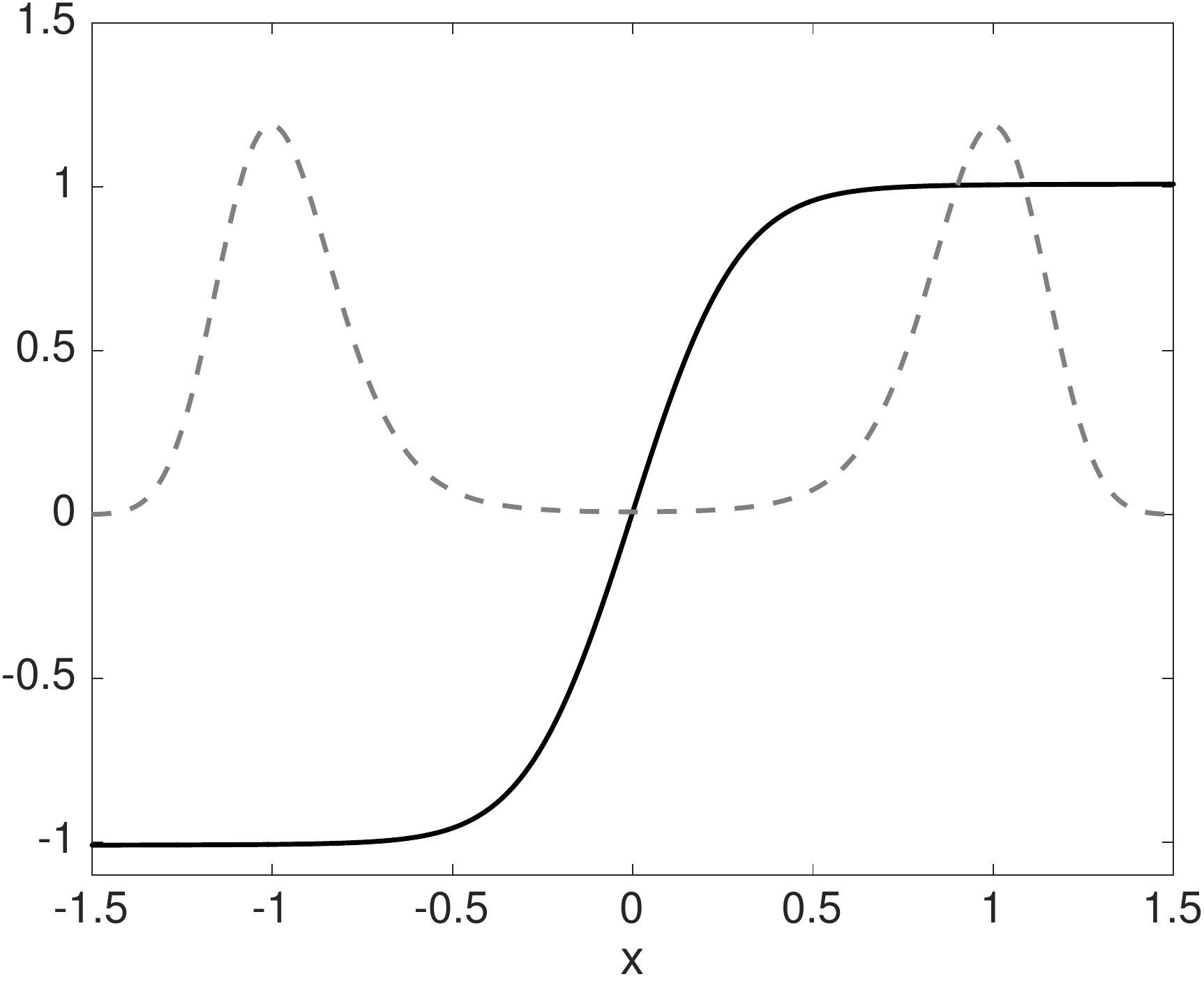}

\caption{Left: double-well potential. Right: invariant distribution $\mu$ (gray dashed) and second eigenfunction~$\phi_2$ (solid black) of the associated transfer operator~$\T$.}
\label{fig:dw_stationary}
\end{figure}
Thus, we expect a rank-2 MSM to recover the dominant time scales very well. Indeed, choosing $\set{M}_1 = (-\infty,0]$ and~$\set{M}_2 = [0,\infty)$ gives~$\phi_2 \approx -\one_{\set{M}_1} + \one_{\set{M}_2}$, and we obtain by~\eqref{eq:met_transrates} that
\[
\widehat{T}_2 =
\begin{pmatrix}
0.943 & 0.057 \\
0.057 & 0.943
\end{pmatrix}.
\]
This is a stochastic matrix with eigenvalues~$\widehat{\lambda}_1 = 1$ and~$\widehat{\lambda}_2 = 0.886$, i.e., yielding an approximate time scale~$\widehat{t}_2 = 82.4$.

\section{Markov state models for time-inhomogeneous systems}
\label{sec:MSMneq}

As all our non-equilibrium cases will be special instances of the most general, time-inhomogeneous case, we consider this next.

\subsection{Minimal propagation error by projections}
\label{ssec:PEneq}

\paragraph{Conceptual changes.}

The above approach to Markov state modeling is relying on the existence of an stationary distribution and reversibility. In the case of a time-inhomogeneous system there will not be, in general, any stationary distribution~$\mu$.
Additionally, we are lacking physical meaning, since it is unclear with respect to which ensemble the dynamical fluctuations should be described.
From a mathematical perspective there is a problem as well, since the construction relies on the reversibility of the underlying system, which gives the self-adjointness of the operator~$\T$ with respect to the weighted scalar product~$\innerprod{\cdot}{\cdot}_\mu$. Time-inhomogeneous systems are not reversible in general.

Additionally to these \emph{structural} properties, we might need to depart from some \emph{conceptional} ones as well. As time-inhomogeneity usually stems from an external forcing that might not be present or known for all times, we need a description of the system on a finite time interval. This disrupts the concept of dominant time scales as they are considered in equilibrium systems, because there it relies on self-similarity of observing an eigenmode over and over for arbitrary large times. It also forces us to re-visit the concept of metastability for two reasons. First, many definitions of metastability rely on statistics under the assumption that we observe the system for infinitely long times. Second, as an external forcing may theoretically arbitrarily distort the energy landscape, it is a priori unclear what could be a \emph{metastable set}.

As a remedy, we aim at another property when trying to reproduce the effective behavior of the full system by a reduced model; this will be minimizing the propagation error, as in~\eqref{eq:densprop}. Remarkably, this will also allow for a physical interpretation through so-called \emph{coherent sets}; analogous to metastable sets in the equilibrium case.

A prototypical time-inhomogeneous system can be given by
\begin{equation}
    \mathrm{d} x_t = -\nabla W(t,x_t)\,\mathrm{d}t + \sqrt{2\beta^{-1}} \, \mathrm{d}w_t\,,
    \label{eq:overdampedLangevinInhom}
\end{equation}
where the potential~$W$ now depends explicitly on time~$t$. In this case, a lag time $\tau$ is not sufficient to parametrize the statistical evolution of the system, because we need to know when we start the evolution. Thus, transition density functions need two time parameters, e.g., $p^{s,t}(x,\cdot)$ denotes the distribution of $x_t$ conditional to~$x_s=x$. Similarly, the transfer operators~$\P,\T,\K$ are parametrized by two times as well, e.g., $\P^{s,t}$ propagates probability densities from initial time $s$ to final time~$t$ (alternatively, from initial time $s$ for lag time~$\tau = t-s$). To simplify notation, we will settle for some initial and final times, and drop these two time parameters, as they stay fixed.

\paragraph{Adapted transfer operators.}

Let us observe the system from initial time $t_0$ to final time~$t_1$, such that its distribution at initial time is given by~$\mu_0$. Then, if~$\P$ denotes the propagator of the system from $t_0$ to $t_1$, then we can express the final distribution at time $t_1$ by~$\mu_1 = \P \mu_0$. As the transfer operator in equilibrium case was naturally mapping~$L^2_\mu$ to itself (because $\mu$ was invariant), here it is natural to consider the transfer operator mapping densities (functions) with respect to~$\mu_0$ to densities with respect to~$\mu_1$. Thus, we define the transfer operator~$\T: L^2_{\mu_0} \to L^2_{\mu_1}$ by
\begin{equation} \label{eq:Tneq}
\T u:= \frac{1}{\mu_1} \P\left(u \mu_0\right),
\end{equation}
which is the non-equilibrium analogue to~\eqref{eq:PT_identity}.
This operator naturally retains some properties of the equilibrium transfer operator~\cite{Den17}:
\begin{compactitem}
\item
$\T\one = \one$, encoding the property that $\mu_0$ is mapped to $\mu_1$ by the propagator $\P$.
\item
$\T$ is positive and integral-preserving, thus $\sigma_{\max}(\T)=1$.
\item
Its adjoint is the Koopman operator~$\K:L^2_{\mu_1}\to L^2_{\mu_0}$, $\K g(x) = \expect[g(x_t)\,\vert x_0=x]$.

\end{compactitem}

\paragraph{An optimal non-stationary GMSM.}

As already mentioned above, it is not straightforward how to address the problem of Markov state modeling in this time-inhomogeneous case via descriptions involving time scales or metastability. Instead, our strategy will be to search for a rank-$k$ projection~$\T_k$ of the transfer operator~$\T$ with minimal propagation error, to be described below.

The main point is now that due to the non-stationarity the \emph{domain} $L^2_{\mu_0}$ (where $\T$ maps from) and \emph{range} $L^2_{\mu_1}$ (where $\T$ maps to) of the transfer operator~$\T$ are different spaces, hence it is natural to choose different rank-$k$ subspaces as domain and range of~$\T_k$ too. In fact, it is necessary to choose domain and range differently, since $f \in L^2_{\mu_0}$ has a different meaning than $f \in L^2_{\mu_1}$. Thus, we will search for projectors~$\Q_0:L^2_{\mu_0}\to\set{V}_0 \subset L^2_{\mu_0}$ and~$\Q_1: L^2_{\mu_1}\to \set{V}_1 \subset L^2_{\mu_1}$ on different $k$-dimensional subspaces~$\set{V}_0$ and~$\set{V}_1$, respectively, such that the reduced operator
\begin{equation} \label{eq:MSMneq}
\T_k := \Q_1\T\Q_0
\end{equation}
has essentially optimal propagation error. In quantitative terms, we seek to solve the optimization problem
\begin{equation} \label{eq:denspropneq}
\T_k = \argmin_{\substack{\T_k' = \Q_1' \T \Q_0' \\ \text{rank}\,\Q_0'=k \\ \text{rank}\,\Q_1'=k}} \max_{\|f\|_{\mu_0} = 1}  \| \T f-\T_k' f \|_{\mu_1}
\quad \text{or, equivalently} \quad
\T_k = \argmin_{\substack{\T_k' = \Q_1' \T \Q_0' \\ \text{rank}\,\Q_0'=k \\ \text{rank}\,\Q_1'=k}}  \| \T-\T_k' \| \,,
\end{equation}
where~$\|\cdot\|$ denotes the induced operator norm of operators mapping~$L^2_{\mu_0}$ to~$L^2_{\mu_1}$.

As an implication of the Eckart--Young theorem~\cite[Theorem 4.4.7]{HsEu15}, the solution of~\eqref{eq:denspropneq} can explicitly be given through singular value decomposition of~$\T$; yielding the variational approach for Markov processes (VAMP)~\cite{WuNo17}. More precisely, the~$k$ largest singular values~$\sigma_1 \ge \ldots \ge \sigma_k$ of~$\T$ have right and left singular vectors~$\phi_i,\psi_i$ satisfying~$\innerprod{\phi_i}{\phi_j}_{\mu_0} = \delta_{ij},  \innerprod{\psi_i}{\psi_j}_{\mu_1} = \delta_{ij}$, respectively, i.e.,~$\T\phi_i = \sigma_i\psi_i$. Choosing
\begin{equation} \label{eq:neqproj}
\Q_0 f = \sum_{i=1}^k \innerprod{\phi_i}{f}_{\mu_0} \phi_i \quad \text{and} \quad \Q_1 g = \sum_{i=1}^k \innerprod{\psi_i}{g}_{\mu_1} \psi_i
\end{equation}
solves~\eqref{eq:denspropneq}, see Appendix~\ref{app:E-Ythm}.

\subsection{Coherent sets}
\label{ssec:cohsets}

Similarly to the reversible equilibrium case with pronounced metastability in section~\ref{ssec:metsets}, it is also possible in the time-inhomogeneous case to give our GMSM~\eqref{eq:MSMneq} from section~\ref{ssec:PEneq} a physical interpretation---under some circumstances.

In the reversible equilibrium situation, recall from~\eqref{eq:EVlincomb} that in the case of sufficient time scale separation the eigenfunctions are almost constant on metastable sets. In the time-inhomogeneous situation, considered now, we have just shown that the role played before by the eigenfunctions is taken by left- and right singular functions. Thus, let us assume for now that there are \emph{two} collections of sets,~$\set{M}_{0,1},\ldots,\set{M}_{0,k}$ at initial time, and~$\set{M}_{1,1},\ldots,\set{M}_{1,k}$ at final time, such that
\begin{equation} \label{eq:SV_lincomb}
\phi_i \approx \sum_{j=1}^k c_{ij}\one_{0,j} =: \widehat{\phi}_i
\qquad\text{and}\qquad
\psi_i \approx \sum_{j=1}^k d_{ij}\one_{1,j} =: \widehat{\psi}_i
\end{equation}
holds with appropriate scalars~$c_{ij}$ and~$d_{ij}$, where we used the abbreviation~$\one_{0,i} = \one_{\set{M}_{0,i}}$ and~$\one_{1,i} = \one_{\set{M}_{1,i}}$. That means, dominant right singular functions~$\phi_i$ are almost constant on the sets~$\set{M}_{0,j}$, and dominant left singular functions~$\psi_i$ are almost constant on the sets~$\set{M}_{1,j}$. In analogy to~\eqref{eq:met_proj}, we modify the projections~$\Q_0,\Q_1$ from~\eqref{eq:neqproj} to~$\widehat{\Q}_0: L^2_{\mu_0} \to \widehat{\set{V}}_0,\widehat{\Q}_1: L^2_{\mu_1} \to \widehat{\set{V}}_1$ by using $\widehat{\phi}_i$ and~$\widehat{\psi}_i$ instead of $\phi_i$ and~$\psi_i$, respectively, and define the modified GMSM by~$\widehat{\T}_k = \widehat{\Q}_1 \T \widehat{\Q}_0$. An analogous computation to~\eqref{eq:met_transrates} yields for the matrix representation~$\widehat{T}_k$ of the restriction $\widehat{\T}_k: \widehat{\set{V}}_0 \to \widehat{\set{V}}_1$ with respect to the bases~$\{ \one_{0,i} / \innerprod{\one_{0,i}}{\one_{0,i}}_{\mu_0} \}_{i=1}^k$ and~$\{ \one_{1,i} / \innerprod{\one_{1,i}}{\one_{1,i}}_{\mu_1} \}_{i=1}^k$ that
\begin{equation} \label{eq:coh_transrates}
\widehat{T}_{k,ij} = \prob_{\mu_0}\left[ x_{t_1} \in \set{M}_{1,i} \,\vert\, x_{t_0} \in \set{M}_{0,j} \right].
\end{equation}
In other words, the entries of $\widehat{T}_k$ contain the transition probabilities from the sets~$\set{M}_{0,i}$ (at initial time) into the sets~$\set{M}_{1,j}$ (at final time). Thus, $\widehat{T}_k$ has the physical interpretation of a MSM, with the only difference to the reversible stationary situation being that the ``metastable'' sets at initial and final time are different. This can be seen as a natural reaction to the fact that in the time-inhomogeneous case the dynamical environment (e.g., the potential energy landscape governing the dynamics of a molecule) can change in time.

\begin{figure}[htb]
\centering

\includegraphics[width = 0.8\textwidth]{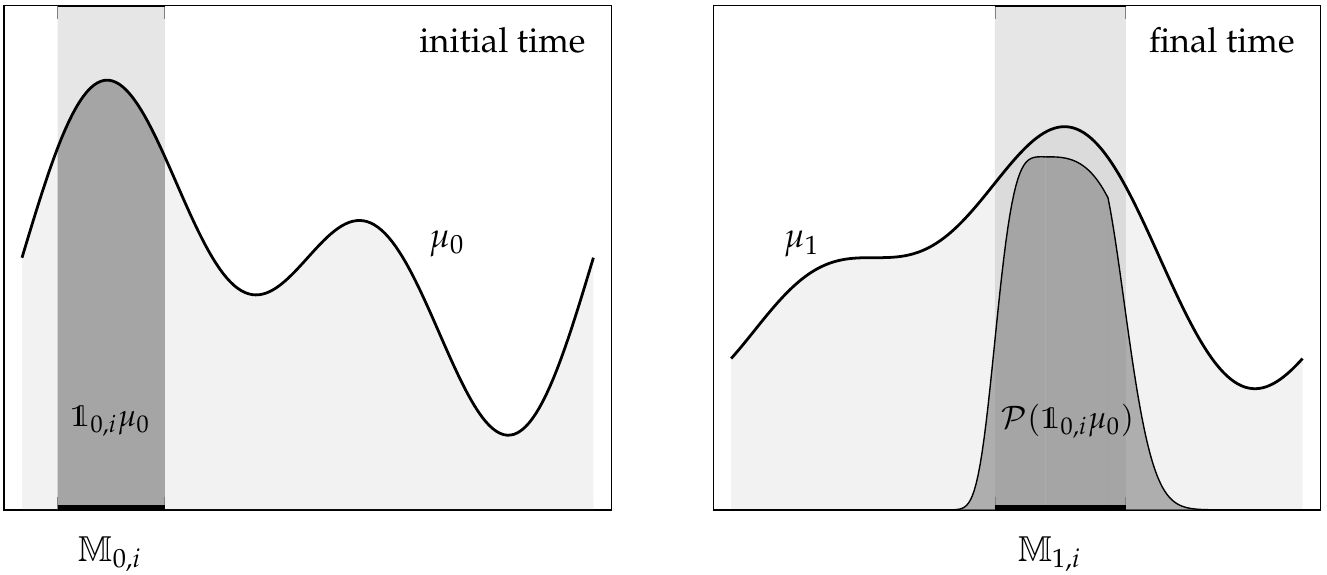}

\caption{Cartoon representation of \eqref{eq:coherence} with a coherent pair~$\set{M}_{0,i}, \set{M}_{1,i}$, which are represented by the thick horizontal lines left and right, respectively. Condition~\eqref{eq:coherence} can be translated into $\T\one_{0,i} \approx \one_{1,i}$, or equivalently $\P(\one_{0,i}\mu_0) \approx \one_{1,i}\mu_1$.
In other words, the part of the ensemble~$\mu_0$ supported on the set $\set{M}_{0,i}$ (dark gray region on the left) is mapped by the propagator to an ensemble (dark gray region on the right) that is almost equal to the part of the ensemble~$\mu_1$ supported on~$\set{M}_{1,i}$. Note that little of $\P(\one_{0,i}\mu_0)$ is supported outside of~$\set{M}_{1,i}$, \emph{and} little of $\one_{1,i}\mu_1$ came from outside~$\set{M}_{0,i}$. }
\label{fig:cohset cartoon}
\end{figure}

It remains to discuss when does~\eqref{eq:SV_lincomb} actually hold true. It is comprehensively discussed in~\cite{KoCiSch16} that a sufficient condition for~\eqref{eq:SV_lincomb} is if
\begin{equation} \label{eq:coherence}
\prob_{\mu_0}\left[ x_{t_1} \in \set{M}_{1,i} \,\vert\, x_{t_0} \in \set{M}_{0,i} \right] \approx 1
\qquad\text{and}\qquad
\prob_{\mu_1}\left[ x_{t_0} \in \set{M}_{0,i} \,\vert\, x_{t_1} \in \set{M}_{1,i} \right] \approx 1
\end{equation}
holds for~$i=1,\ldots,k$. Eq.~\eqref{eq:coherence} says that if the process starts in $\set{M}_{0,i}$, it ends up at final time with high probability in~$\set{M}_{1,i}$, and that if the process ended up in~$\set{M}_{1,i}$ at final time, in started with high probability in~$\set{M}_{0,i}$; see Figure~\ref{fig:cohset cartoon}. This can be seen as a generalization of the metastability condition from section~\ref{ssec:metsets} that allows for an efficient low-rank Markov modeling in the time-homogeneous case. The pairs of sets~$\set{M}_{0,i}, \set{M}_{1,i}$ are called \emph{coherent (set) pair}, and they have been shown to be very effective tools identifying \emph{time-dependent} regions in non-autonomous flow fields that do not mix with their surrounding\footnote{This is, effectively, what~\eqref{eq:coherence} says.}, e.g., moving vortices in atmospheric and oceanographic applications~\cite{FrSaMo10,FHRSS12,Froyland2013,FrEtAl15}. More details on the generalization of the concept of metastability by coherent sets, and on subsequent Markov state modeling can be found in~\cite{KoCiSch16}.

\subsection{Example: diffusion in shifting triple-well potential}

Let us consider the diffusion~\eqref{eq:overdampedLangevin} in the time-dependent potential landscape
\[
W(t,x) = 7\left( (x-t/10)(x-1-t/10)(x+1-t/10)\right)^2
\]
with~$\beta = 5$ and on the time interval~$[t_0,t_1] = [0,10]$; cf.~Figure~\ref{fig:tw_shift_potevecs} (left). Taking the initial distribution~$\mu_0 \propto \exp(-\beta W(0,\cdot))$, we build the transfer operator~\eqref{eq:Tneq}, and consider its dominant singular values:
\[
\sigma_1 = 1.000,\quad
\sigma_2 = 0.734,\quad
\sigma_3 = 0.536,\quad
\sigma_4\approx 0\,.
\]
This indicates that a rank-3 GMSM is sufficient to approximate the system, and that we have three coherent sets. We observe the characteristic almost constant behavior~\eqref{eq:SV_lincomb} of the left and right singular vectors over the respective coherent sets; Figure~\ref{fig:tw_shift_potevecs} (middle and right). Recall that right singular vectors show coherent sets at initial time, and left singular vectors the associated coherent sets at final time.

\begin{figure}[htb]
\centering

\includegraphics[width = 0.335\textwidth]{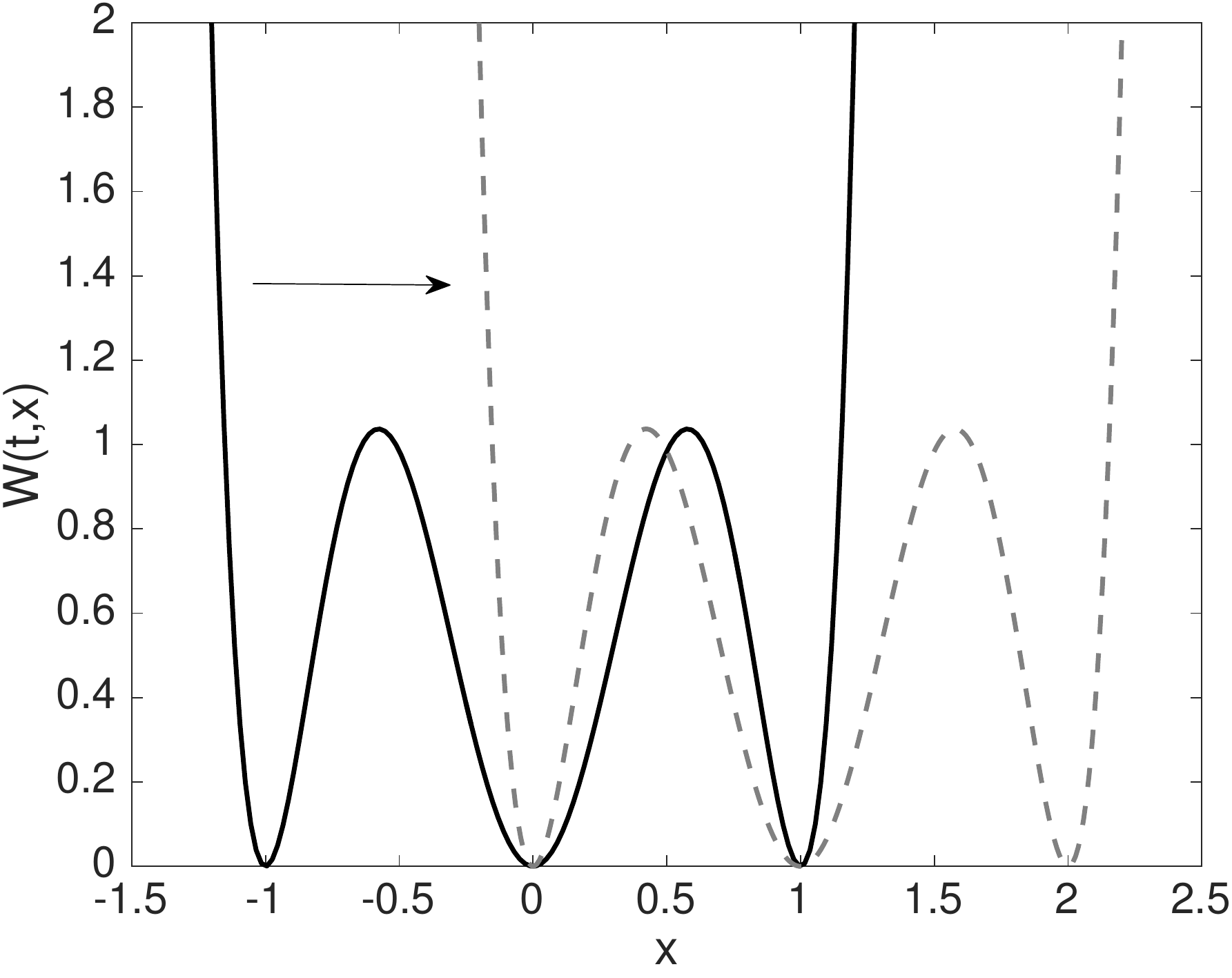}
\hfill
\includegraphics[width = 0.32\textwidth]{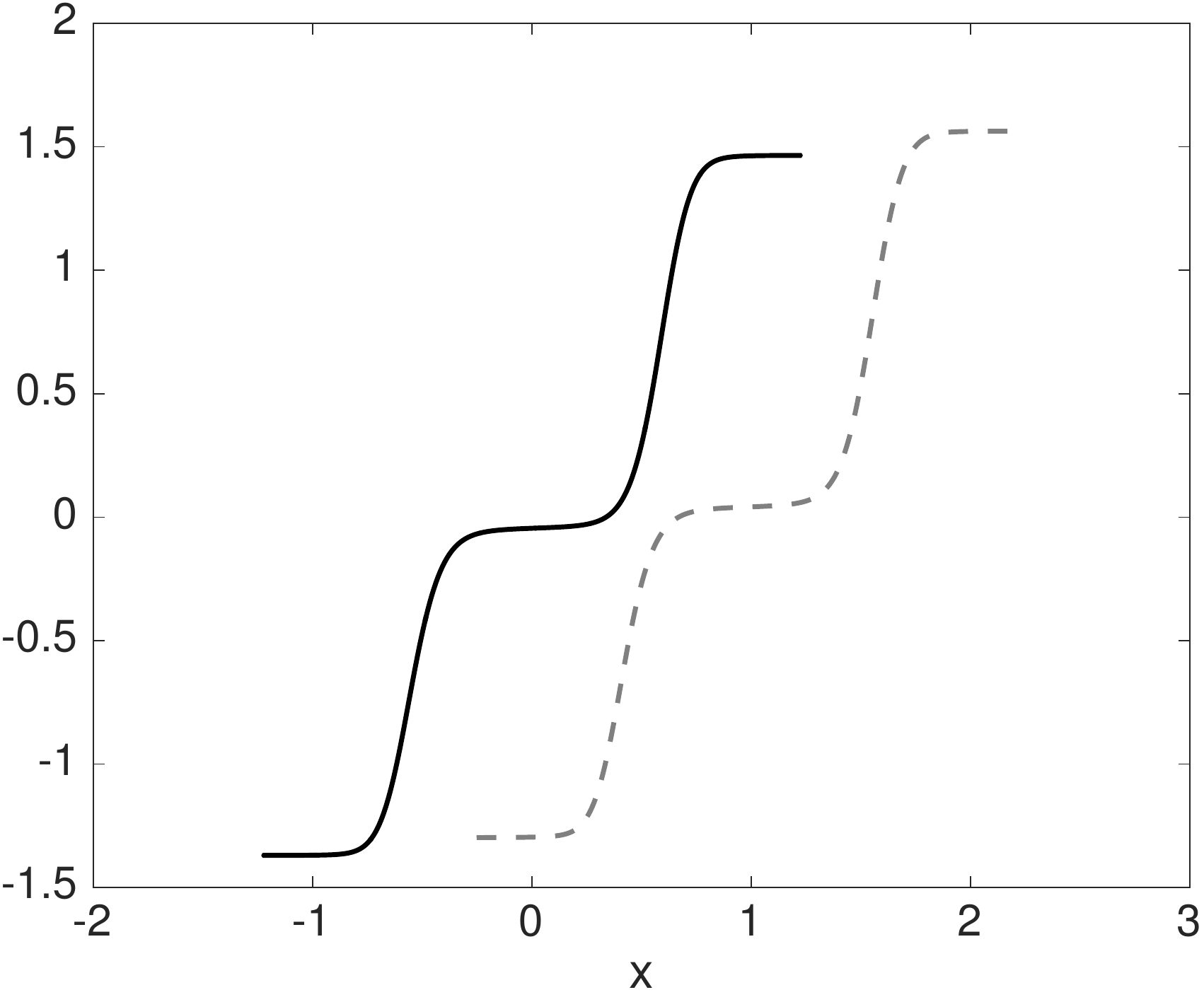}
\hfill
\includegraphics[width = 0.32\textwidth]{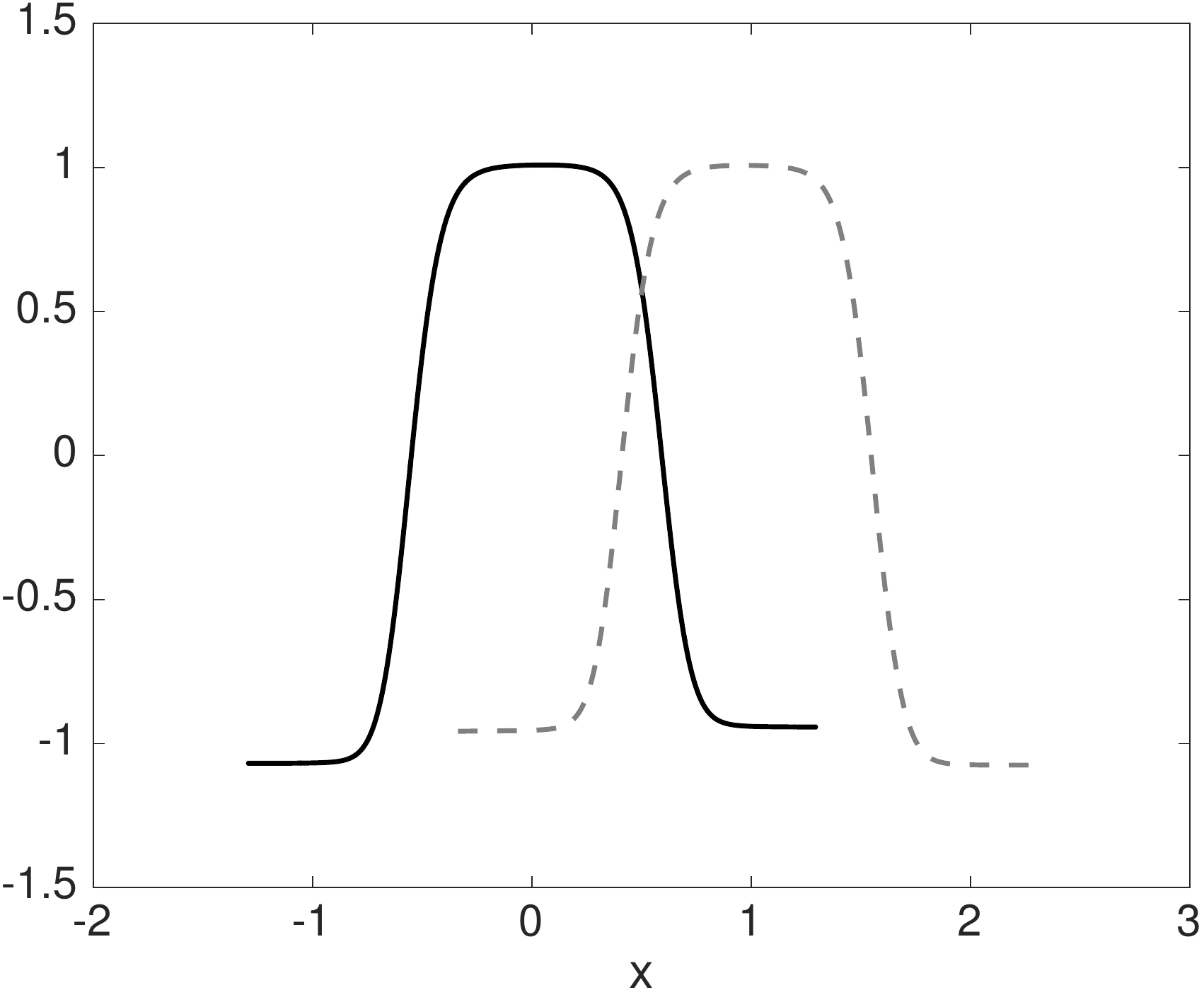}

\caption{Left: shifting triple-well potential. All three wells are coherent sets, as the plateaus of the singular vectors indicate. Middle: second right (initial) and left (final) singular vectors of the transfer operator (solid black and gray dashed lines, respectively). Right: third right (initial) and left (final) singular vectors of the transfer operator (solid black and gray dashed lines, respectively). The singular vectors are for reasons of numerical stability only computed in regions where $\mu_0$ and $\mu_1$ are, respectively, larger than machine precision.}
\label{fig:tw_shift_potevecs}
\end{figure}

We can identify the three wells as three coherent sets. Figure~\ref{fig:tw_shift_ensembles} shows that they are coherent indeed: the respective parts of the initial ensemble~$\mu_0$ is to a large extent mapped onto the corresponding part of the final ensemble~$\mu_1$, cf.~Figure~\ref{fig:cohset cartoon} and~\eqref{eq:coherence}.
\begin{figure}[htb]
\centering

\includegraphics[width = 0.3\textwidth]{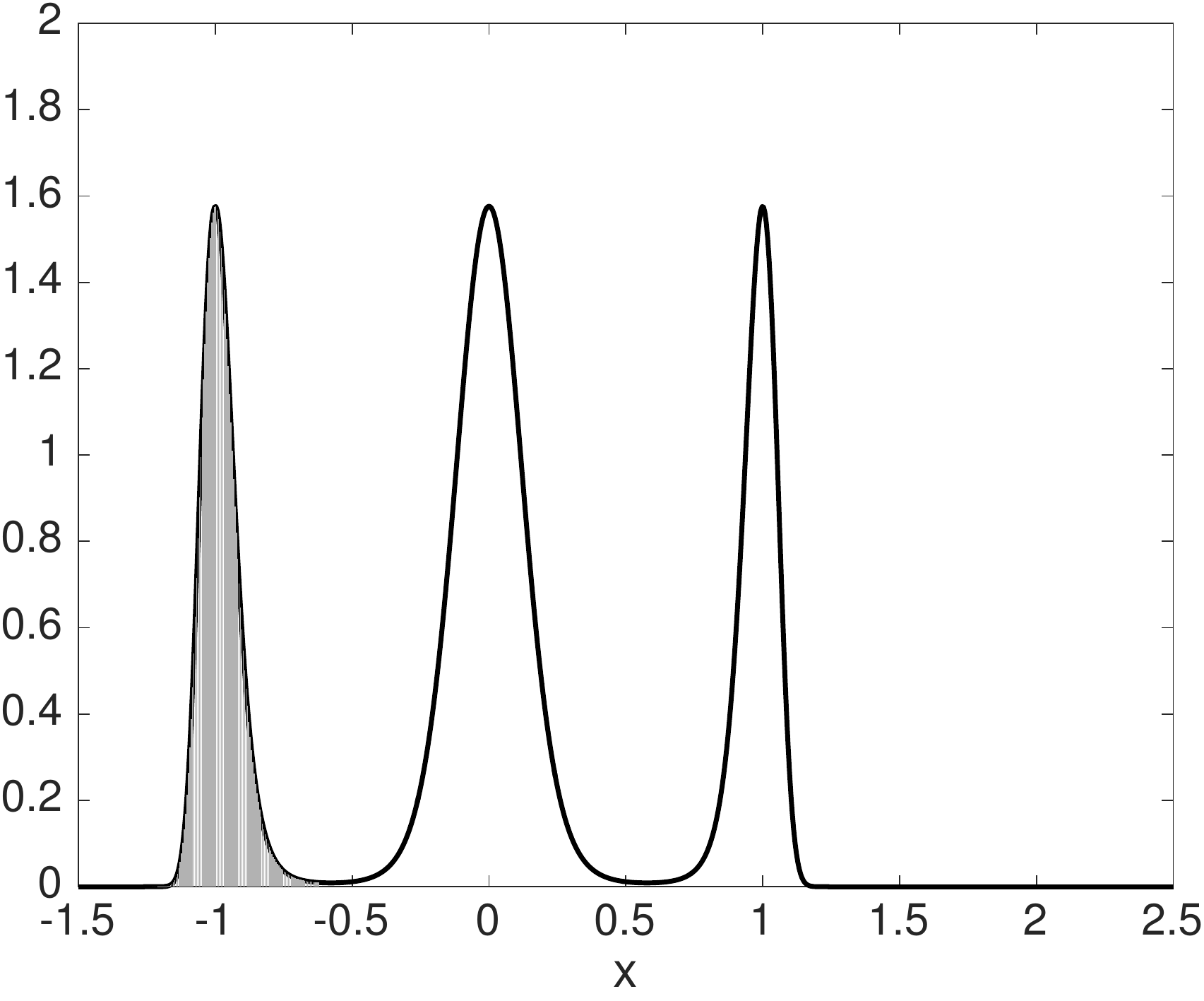}
\hfill
\includegraphics[width = 0.3\textwidth]{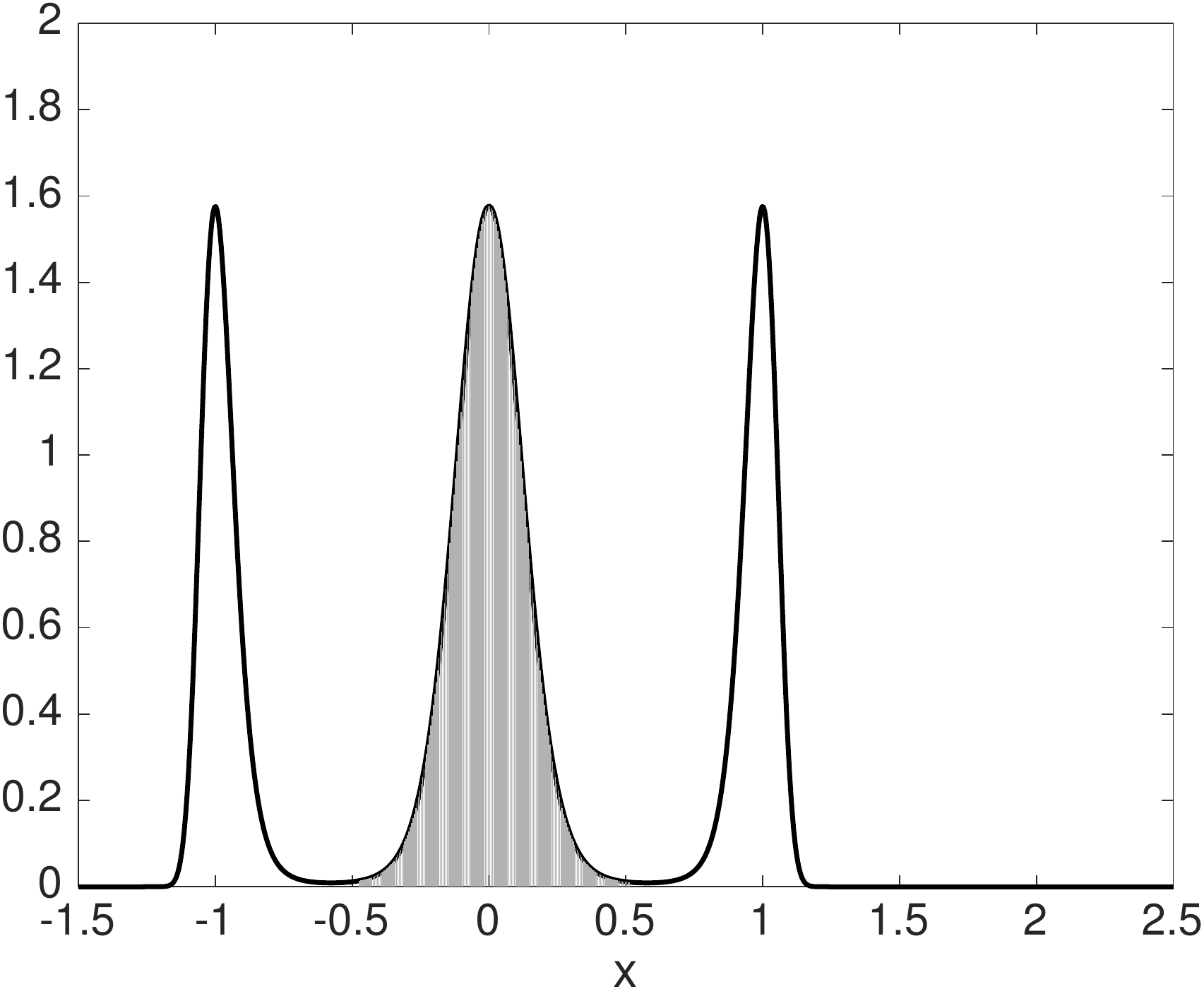}
\hfill
\includegraphics[width = 0.3\textwidth]{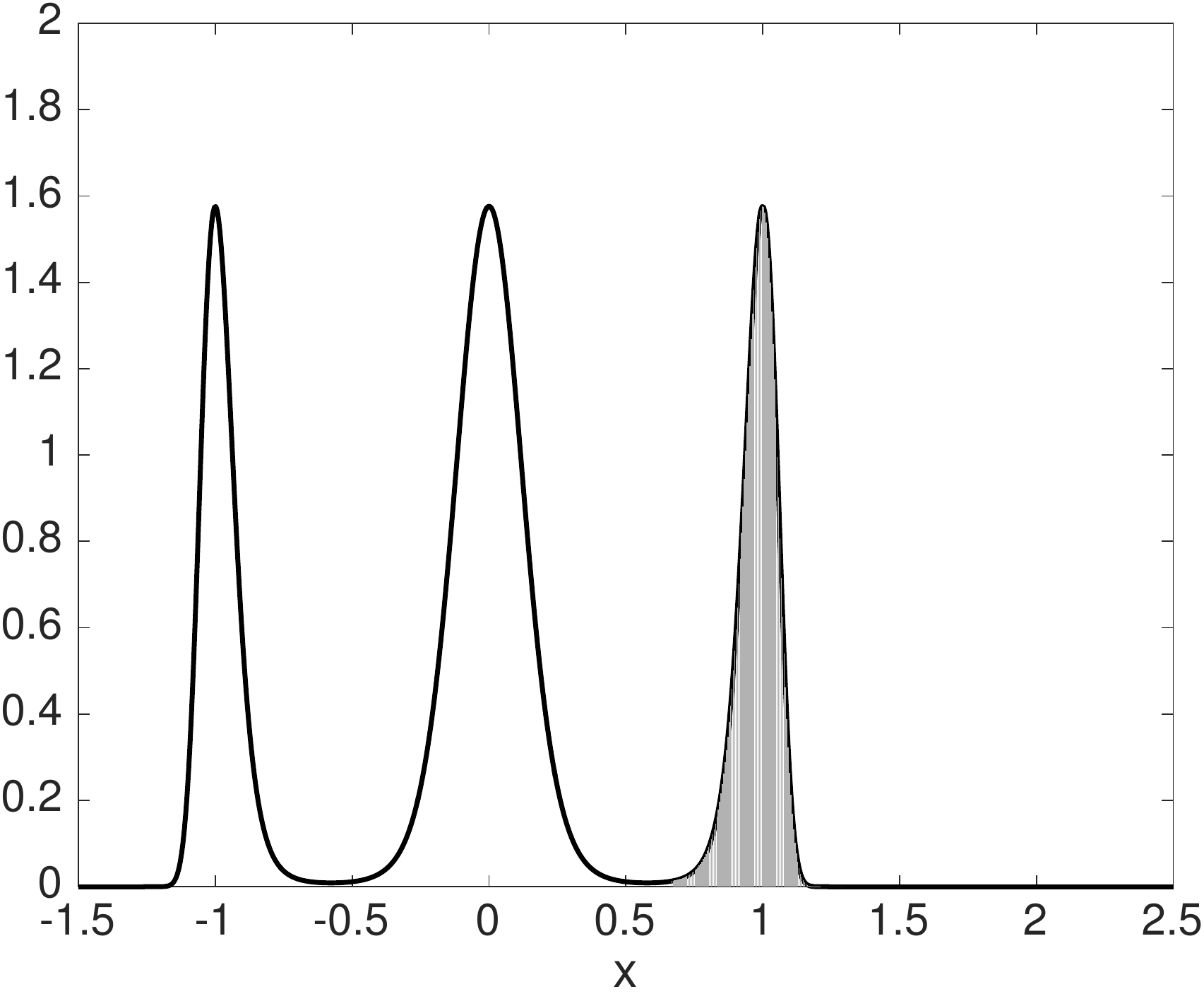}

\includegraphics[width = 0.3\textwidth]{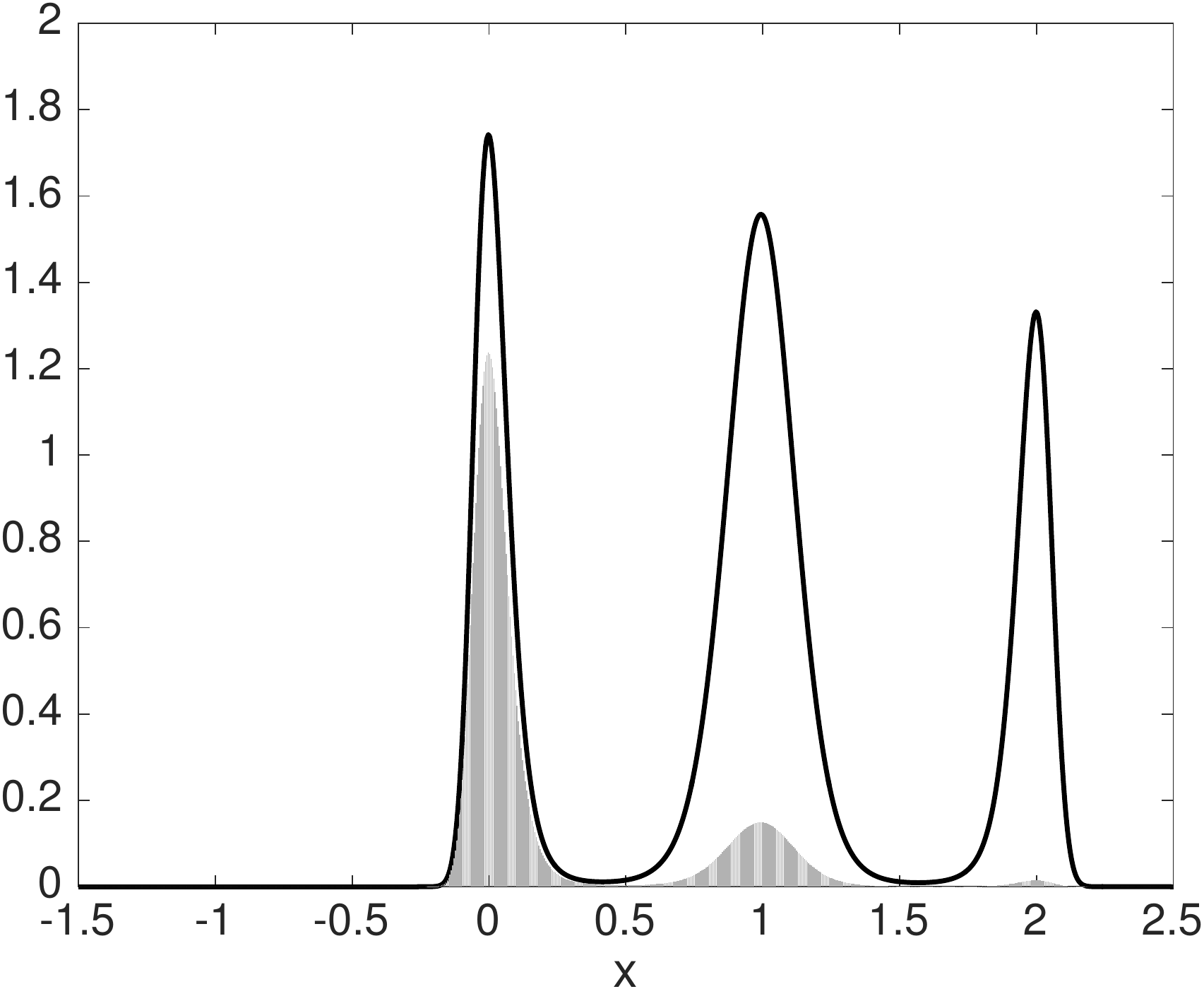}
\hfill
\includegraphics[width = 0.3\textwidth]{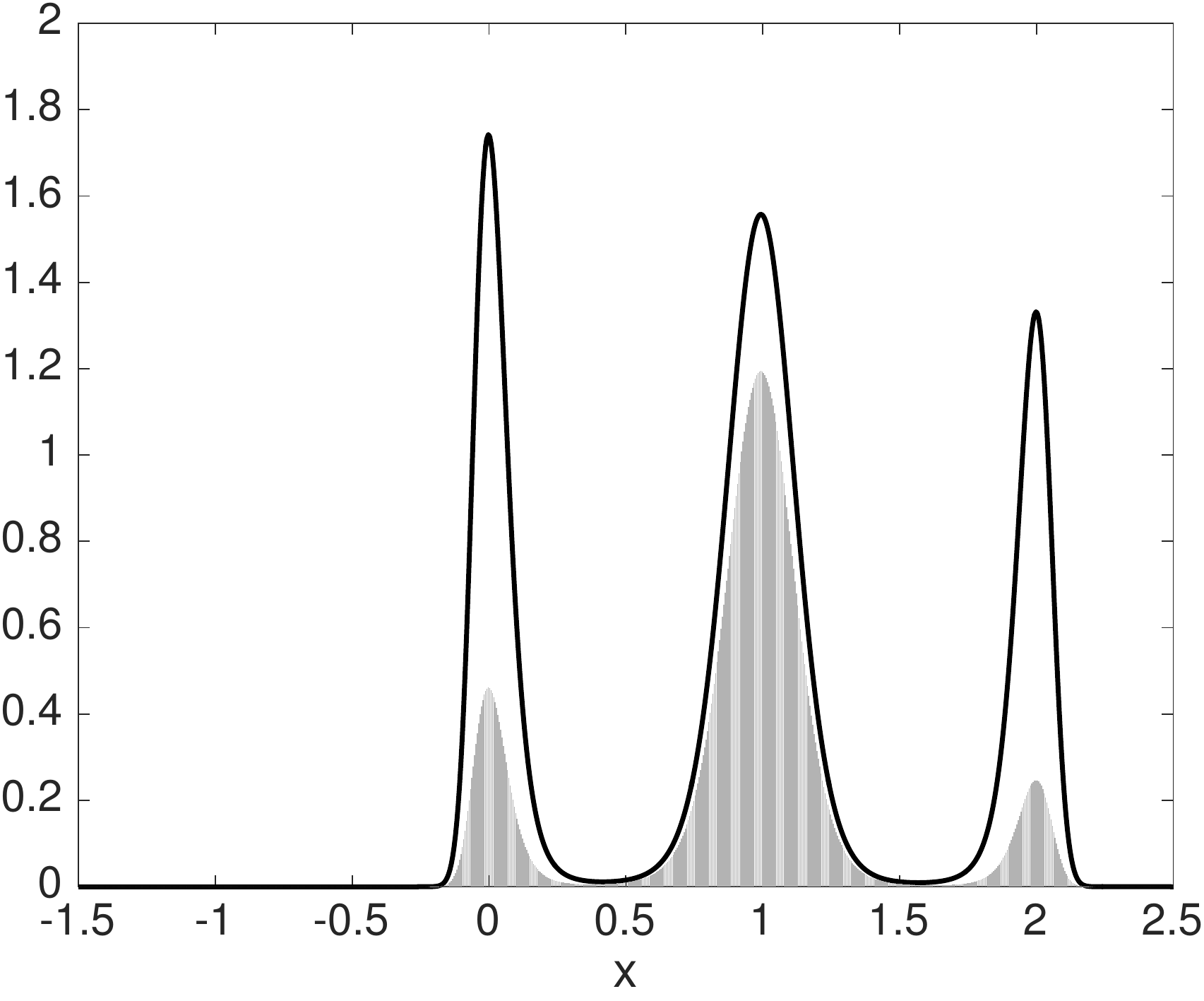}
\hfill
\includegraphics[width = 0.3\textwidth]{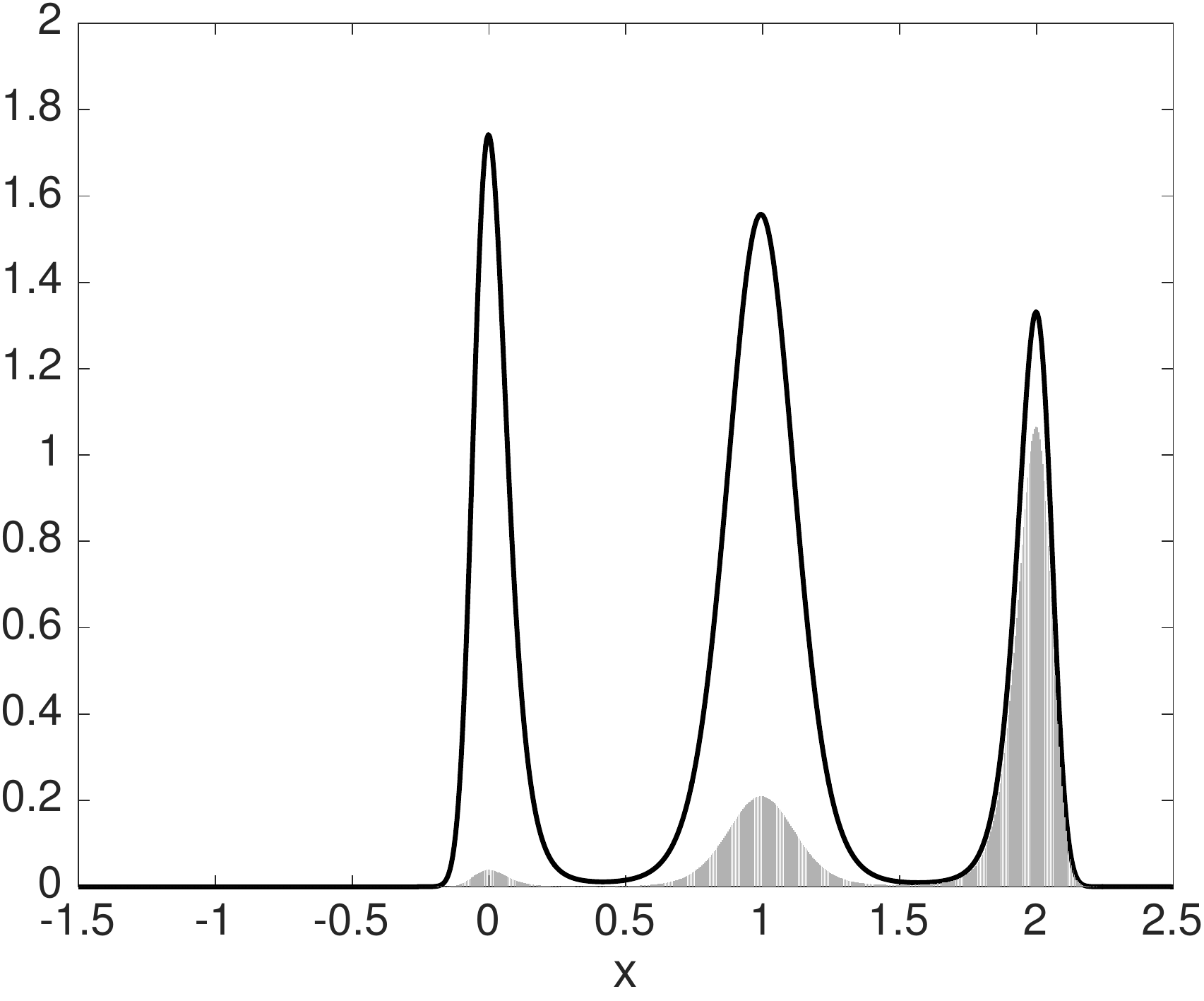}

\caption{Top: initial ensemble $\mu_0$ (black solid) and its respective parts in the three coherent sets (gray shading). Bottom: final ensemble $\mu_1$ (black solid) and the image of the corresponding gray ensembles from the top row (gray shading).}
\label{fig:tw_shift_ensembles}
\end{figure}
Computing the MSM from the transition probabilities between the coherent sets as in~\eqref{eq:coh_transrates} gives the stochastic matrix
\[
\widehat{T}_3 =
\begin{pmatrix}
0.794 & 0.150 & 0.026 \\
0.196 & 0.767 & 0.274 \\
0.010 & 0.083 & 0.701
\end{pmatrix}.
\]
The initial distribution~$\widehat{\mu}_0$ of this MSM is given by the probability that~$\mu_0$ assigns to the respective coherent sets at initial time. Analogously, collecting the probabilities from~$\mu_1$ in the coherent sets at final time gives the final distribution $\widehat{\mu}_1$ of the MSM. We have
\[
\widehat{\mu}_0 = \begin{pmatrix}
0.250 \\ 0.500\\ 0.250
\end{pmatrix}
\quad\text{and}\quad
\widehat{\mu}_1 = \begin{pmatrix}
0.280 \\ 0.500 \\ 0.219
\end{pmatrix}.
\]
The singular values of $\widehat{T}_k$ as mapping from the $\widehat{\mu}_0$-weighted $\R^3$ to the $\widehat{\mu}_1$-weighted $\R^3$ are
\[
\widehat{\sigma}_1 = 1.000,\quad
\widehat{\sigma}_2 = 0.733,\quad
\widehat{\sigma}_3 = 0.534\,;
\]
they are in good agreement with the true singular values of~$\T$.

We repeat the computation with a different initial distribution~$\mu_0$, where only the left and right well are initially populated, as shown in Figure~\ref{fig:tw_shift_ensembles_rholeftmid}.
\begin{figure}[htb]
\centering

\includegraphics[width = 0.3\textwidth]{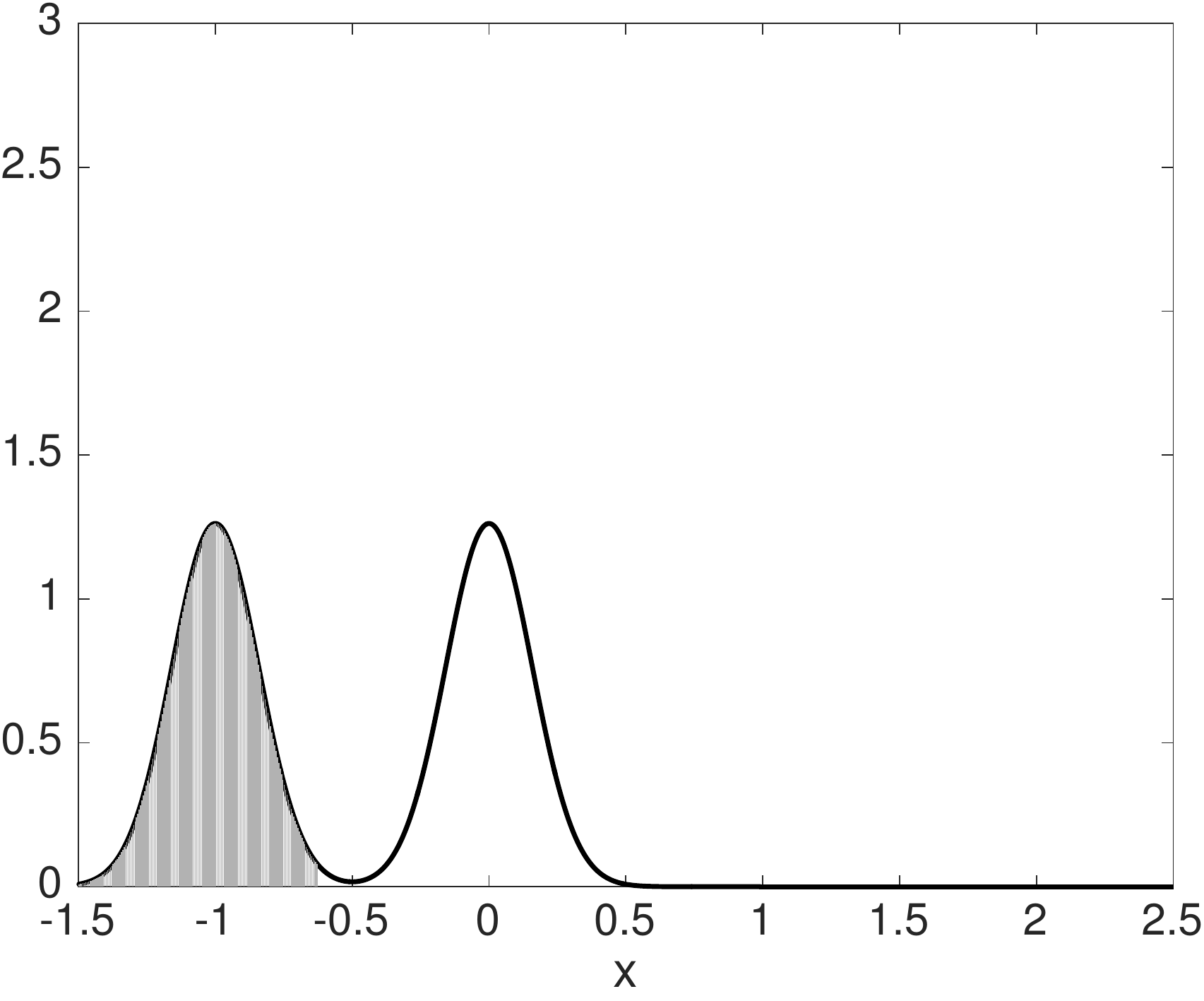}
\qquad
\includegraphics[width = 0.3\textwidth]{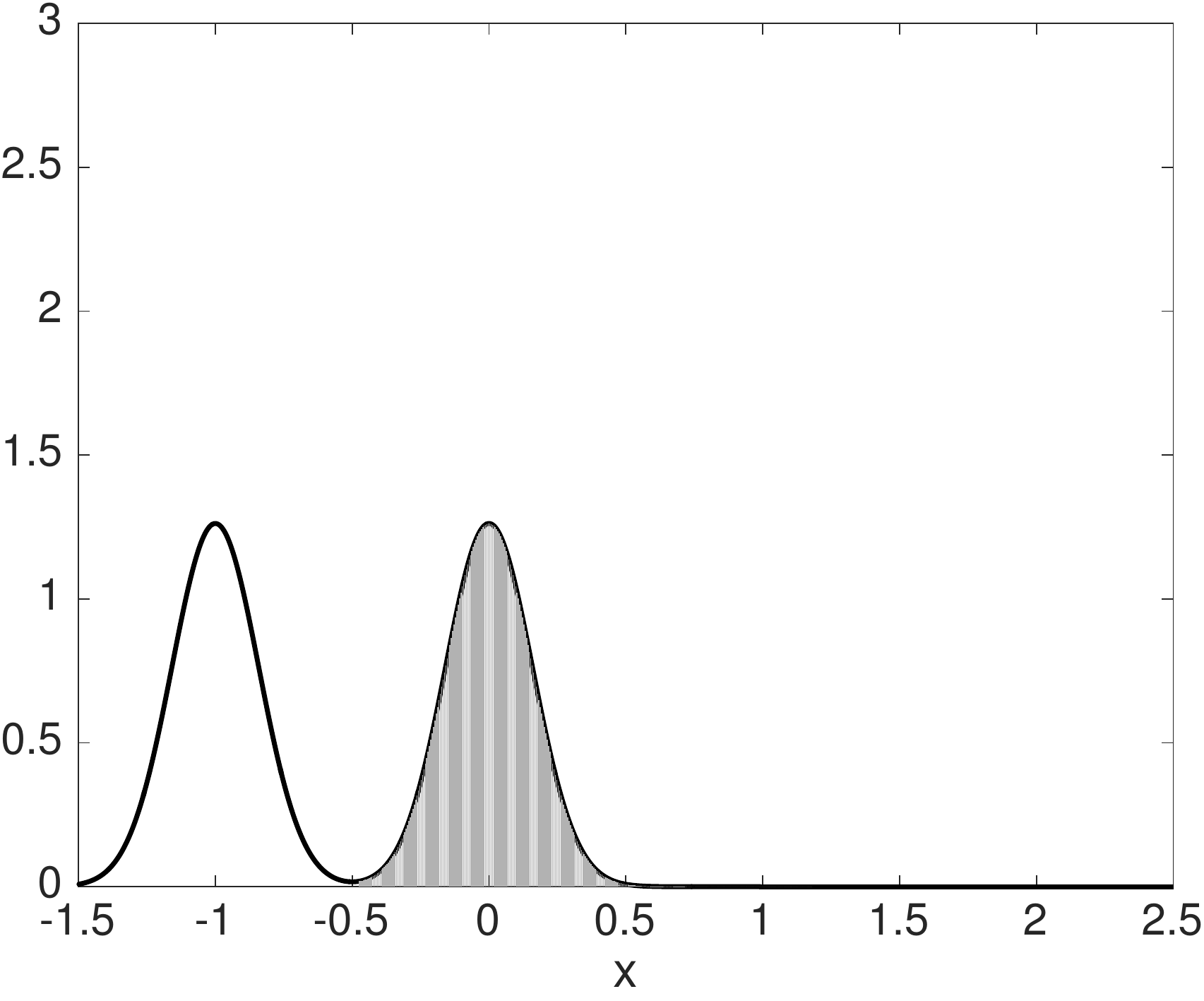}

\includegraphics[width = 0.3\textwidth]{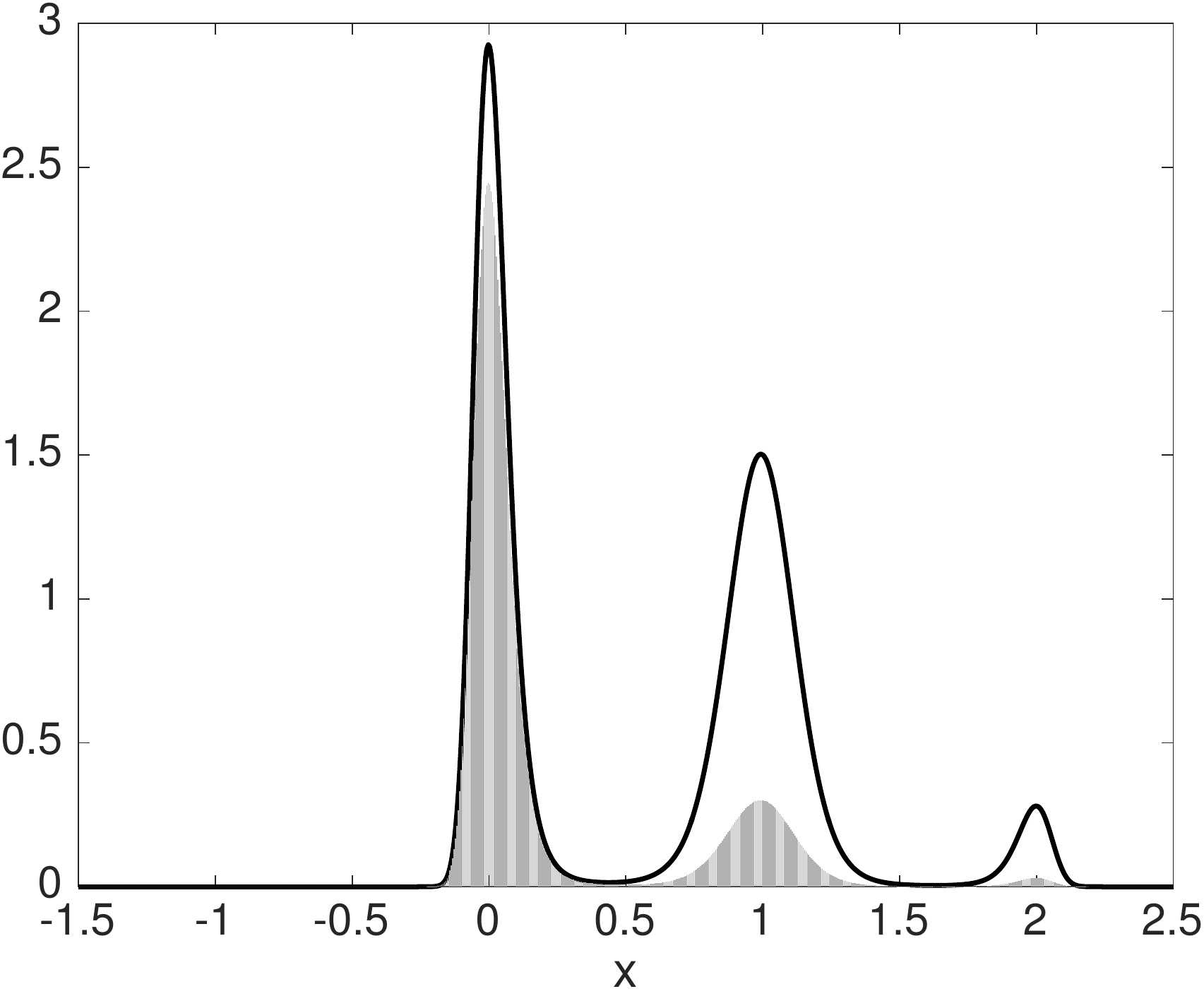}
\qquad
\includegraphics[width = 0.3\textwidth]{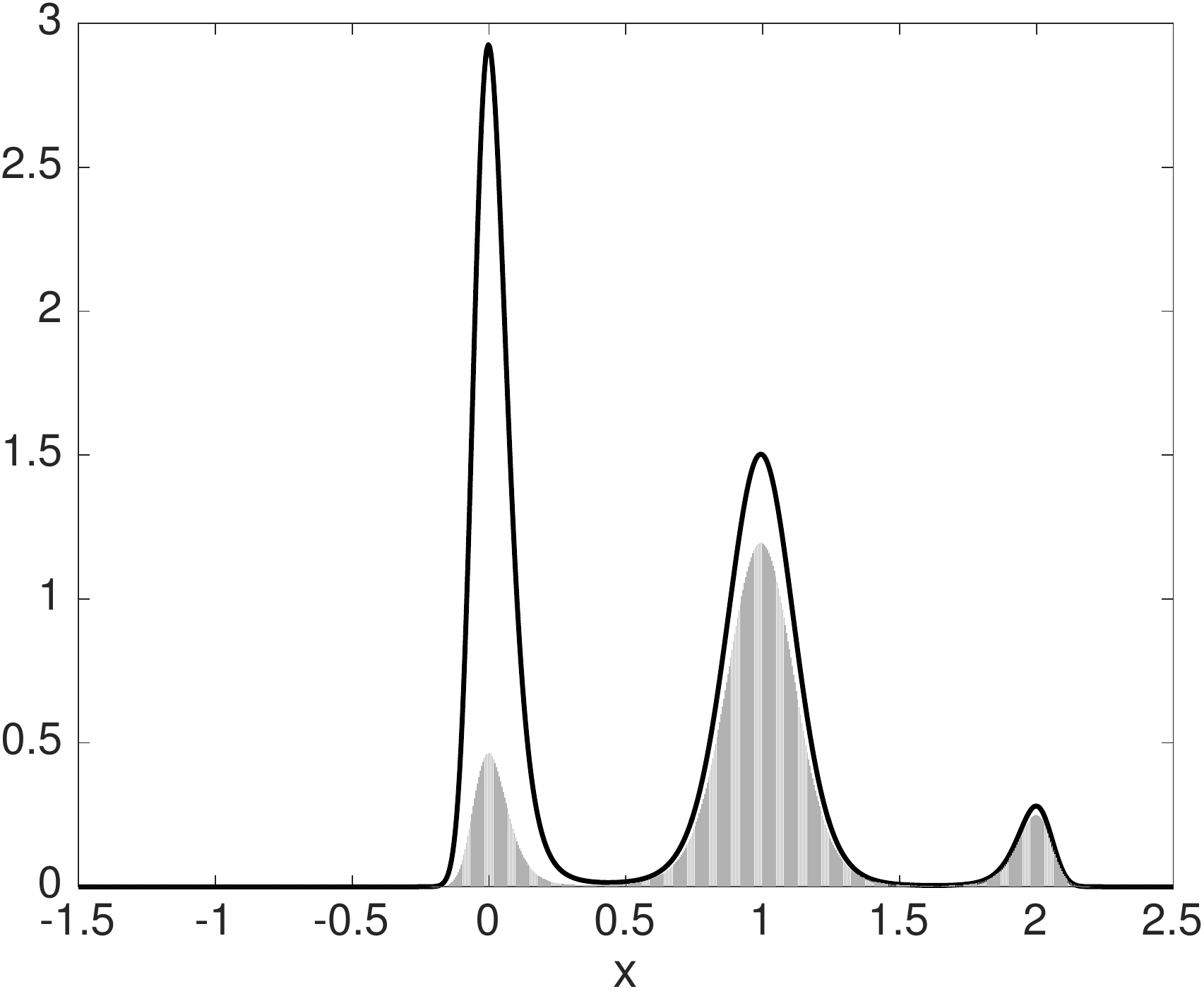}

\caption{The same as Figure~\ref{fig:tw_shift_ensembles}, for a different initial distribution $\mu_0$.}
\label{fig:tw_shift_ensembles_rholeftmid}
\end{figure}
The largest singular values of~$\T$,
\[
\sigma_1 = 1.000,\quad
\sigma_2 = 0.643,\quad
\sigma_3 = 0.030,\quad
\sigma_4\approx 0\,,
\]
already show that there are only two coherent sets, as the third singular value is significantly smaller than the second one. The left well forms one coherent set, and the union of the middle and right ones form the second coherent set.

\section{Data-based approximation}
\label{sec:data based approx}

\paragraph{Setting and auxiliary objects.}

We would like to estimate the GMSM~\eqref{eq:MSMneq} from trajectory data. In the time-inhomogeneous  setting, let us assume that we have $m$ data points~$x_1,\ldots,x_m$ at time $t_0$, and their (random) images $y_1,\ldots,y_m$ at time $t_1$. We can think of the empirical distribution of the $x_i$ and $y_i$ being estimates of $\mu_0$ and $\mu_1$, respectively.

Let us further define two sets of basis functions $\chi_{0,1},\ldots,\chi_{0,n}$ and $\chi_{1,1},\ldots,\chi_{1,n}$, which we would like to use to approximate the GMSM. If we would like to estimate the first $k$ dominant modes, the least requirement is $n\ge k$; in general we have $n\gg k$. The vector-valued functions
\[
\chi_0 = \begin{pmatrix}
\chi_{0,1}\\ \vdots \\ \chi_{0,n}
\end{pmatrix},\quad
\chi_1 = \begin{pmatrix}
\chi_{1,1}\\ \vdots \\ \chi_{1,n}
\end{pmatrix}
\]
are basis functions at initial and final times, respectively. One can take $\chi_0$ and $\chi_1$ to have different lengths too, we just chose them to have the same lengths for convenience. Now we can define the data matrices
\[
\dat{\chi}_0 = \begin{pmatrix}
\vert & & \vert \\
\chi_0(x_1) & \ldots & \chi_0(x_m) \\
\vert & & \vert
\end{pmatrix},\quad
\dat{\chi}_1 = \begin{pmatrix}
\vert & & \vert \\
\chi_1(y_1) & \ldots & \chi_1(y_m) \\
\vert & & \vert
\end{pmatrix}\,.
\]
The following $n\times n$ correlation matrices~$C_{00}, C_{01}, C_{11}$  will be needed later:
\[
C_{00,ij} = \innerprod{\chi_{0,i}}{\chi_{0,j}}_{\mu_0} ,\quad C_{01,ij} = \innerprod{\chi_{1,i}}{\T \chi_{0,j}}_{\mu_1}, \quad
C_{11,ij} = \innerprod{\chi_{1,i}}{\chi_{1,j}}_{\mu_1}\,.
\]
Their Monte Carlo estimates from the trajectory data are given by products of the data-matrices, as
\begin{equation} \label{eq:corr_mats}
C_{00} \approx \frac{1}{m} \dat{\chi}_0\dat{\chi}_0^T,\quad
C_{01} \approx \frac{1}{m} \dat{\chi}_1\dat{\chi}_0^T,\quad
C_{11} \approx \frac{1}{m} \dat{\chi}_1\dat{\chi}_1^T\,.
\end{equation}
Note that the approximations in~\eqref{eq:corr_mats} become exact if we take $\mu_0,\mu_1$ to be the empirical distributions~$\mu_0 = \tfrac{1}{m}\sum_i \delta(\cdot-x_i)$ and~$\mu_1 = \tfrac{1}{m}\sum_i \delta(\cdot-y_i)$, where~$\delta(\cdot)$ denotes the Dirac delta.
We assume that $C_{00}, C_{11}$, just as their data-based approximations in~\eqref{eq:corr_mats} are invertible. If they are not, all the occurrences of their inverses below need to be replaced by their Moore--Penrose pseudoinverses. Alternatively, one can also discard basis functions that yield redundant information, until~$C_{00}, C_{11}$ are invertible. Further strategies to deal with the situation where the correlation matrices are singular or ill-conditioned can be found in~\cite{WuEtAl17}.

\paragraph{Projection on the basis functions.}

To find the best GMSM representable with the bases $\chi_0$ and $\chi_1$, we would like to solve \eqref{eq:denspropneq} under the constraint that the ranges of~$\Q_0$ and $\Q_1$ are in $\set{W}_0:=\text{span}(\chi_0)$ and~$\set{W}_1:=\text{span}(\chi_1)$, respectively. To the knowledge of the authors it is unknown whether this problem has an explicitly computable solution, because it involves a non-trivial interaction of~$\set{W}_0,\set{W}_1$ and~$\T$.

Instead, we will proceed in two steps. First, we compute the projected transfer operator~$\T_n = \Pi_1\T\Pi_0$, where~$\Pi_0$ and~$\Pi_1$ are the $\mu_0$- and~$\mu_1$-orthogonal projections on~$\set{W}_0$ and~$\set{W}_1$, respectively. Second, we reduce~$\T_n$ to its best rank-$k$ approximation~$\T_k$ (best in the sense of density propagation).

Thus, the restriction~$\T_n$ to~$\set{W}_0\to \set{W}_1$ is simply the $\mu_1$-orthogonal projection of $\T$ on~$\set{W}_1$, giving the characterization
\begin{equation} \label{eq:T_n_Galerkin}
\innerprod{\chi_{1,j}}{\T\chi_{0,i} - \T_n\chi_{0,i}}_{\mu_1} = 0,\quad \forall i,j\,.
\end{equation}
It is straightforward to compute that with respect to the bases $\chi_0$ and~$\chi_1$ the matrix representation~$T_n$ of~$\T_n: \set{W}_0\to \set{W}_1$ is given by
\begin{equation} \label{eq:T_N_repres}
T_n = C_{11}^{-1}C_{01}\,,
\end{equation}
see~\cite{WuNo17}.

\paragraph{Best low-rank approximation.}

To find the best rank-$k$ projection of~$\T_n$, let us now switch to the bases $\tilde{\chi}_0 = C_{00}^{-1/2}\chi_0$ and $\tilde{\chi}_1 = C_{11}^{-1/2}\chi_1$. We can switch between representations with respect the these bases by
\[
f = \sum_{k=1}^n c_k \chi_{0,k}\quad \Longleftrightarrow \quad f = \sum_{k=1}^n \tilde{c}_k \tilde{\chi}_{0,k},\text{ where }\tilde{c} = C_{00}^{1/2}c\,,
\]
and similarly for~$\chi_1$ and~$\tilde{\chi}_1$. Again, a direct calculation shows that $\tilde{\chi}_0$ and $\tilde{\chi}_1$ build orthonormal bases, i.e., $\innerprod{\tilde{\chi}_{0,i}}{\tilde{\chi}_{0,j}}_{\mu_0} = \delta_{ij}$ and $\innerprod{\tilde{\chi}_{1,i}}{\tilde{\chi}_{1,j}}_{\mu_1} = \delta_{ij}$. This has the advantage, that for any operator~$\op{S}_n:\set{W}_0 \to \set{W}_1$ having matrix representation~$S_n$ with respect to the bases $\tilde{\chi}_0$ and~$\tilde{\chi}_1$ we have
\begin{equation}
\|\op{S}_n\| = \|S_n\|_2,
\end{equation}
where $\|\cdot\|_2$ denotes the spectral norm of a matrix (i.e., the matrix norm induced by the Euclidean vector norm). The matrix representation of~$\T_n$ in the new bases is
\begin{equation} \label{eq:T_n_newbasis}
\tilde{T}_n = C_{11}^{-1/2} C_{01} C_{00}^{-1/2}\,.
\end{equation}
However, finding now the best rank-$k$ approximation~$\T_k$ of $\T_n$ amounts, written in these new bases, to
\[
\| \T_n - \T_k\| = \| \tilde{T}_n - \tilde{T}_k\|_2 \to \min_{\text{rank}(\tilde{T}_k) = k}.
\]
Again, by the Eckart--Young theorem~\cite[Theorem 4.4.7]{HsEu15}, the solution to this problem is given by
\begin{equation} \label{eq:T_k}
\tilde{T}_k = \tilde{V}\tilde{\Sigma}\tilde{U}^T,
\end{equation}
where~$\tilde{U}, \tilde{V}\in\R^{n\times k}$ are the matrices with columns being the right and left singular vectors of $\tilde{T}_n$ to the largest $k$ singular values~$\sigma_1\ge \ldots\ge \sigma_k$, and~$\tilde{\Sigma}$ is the diagonal matrix with these singular values on its diagonal. Thus, the best GMSM in terms of propagation error is given with respect to the bases~$\chi_0$ and~$\chi_1$ by
\begin{equation} \label{eq:bestMSM}
T_k = C_{11}^{-1/2} \tilde{V}\tilde{\Sigma}\tilde{U}^T C_{00}^{1/2}\,.
\end{equation}
The resulting algorithm to estimate the optimal GMSM is now identical to the time-lagged canonical correlation algorithm (TCCA) that results from VAMP and is described in~\cite{WuNo17}.
\begin{algorithm} \label{algo:bestMSM}
\begin{enumerate}
\item\setlength{\itemsep}{0mm}
Choose bases $\chi_0$ and~$\chi_1$.
\item
Estimate the correlation matrices $C_{00}, C_{01}, C_{11}$ from data, as in~\eqref{eq:corr_mats}.
\item
Build the projection~$\tilde{T}_n$ of the transfer operator with respect to the modified bases~$\tilde{\chi}_0 = C_{00}^{-1/2}\chi_0$ and $\tilde{\chi}_1 = C_{11}^{-1/2}\chi_1$, i.e., $\tilde{T}_n = C_{11}^{-1/2} C_{01} C_{00}^{-1/2}$, cf.~\eqref{eq:T_n_newbasis}.
\item
Compute the $k$ largest singular values and corresponding right and left singular vectors of~$\tilde{T}_n$, collected into the matrices~$\tilde{\Sigma}$ and~$\tilde{U}, \tilde{V}$, respectively.
\item
The optimal rank-$k$ GMSM has with respect to the original bases~$\chi_0$ and~$\chi_1$ the matrix representation~$C_{11}^{-1/2} \tilde{V}\tilde{\Sigma}\tilde{U}^T C_{00}^{1/2}$; cf.~\eqref{eq:bestMSM}.
\end{enumerate}
\vspace*{-0.5\baselineskip}\caption{TCCA algorithm to estimate a rank-$k$ GMSM.}
\end{algorithm}

\begin{remark}[Reversible system with equilibrium data] \label{rem:revsys}
If the system in consideration is reversible, the data samples its equilibrium distribution, i.e.,~$\mu_0 = \mu_1 = \mu$, and also~$\chi_0 = \chi_1$, then~$C_{00} = C_{11}$, and by the self-adjointness of~$\T$ from~\eqref{eq:selfadj} we have~$C_{01} = C_{01}^T$. Thus,~$\tilde{T}_n$ in~\eqref{eq:T_n_newbasis} is a symmetric matrix, and as such, its singular value and eigenvalue decompositions coincide. Hence, the construction for the best GMSM in this section (disregarding the projection on the basis functions) coincides with the one from section~\ref{sec:MSMeq}. This is not surprising, as both give the best model in terms of propagation error.
\end{remark}

\begin{remark}[Other data-based methods] \label{rem:datameth}
The approximation~\eqref{eq:T_N_repres} of the transfer operator has natural connections to other data-based approximation methods. It can be seen as a problem-adapted generalization of the so-called Extended Dynamic Mode Decomposition (EDMD)~\cite{WKR15,KlKoSch16}. Strictly speaking, however, EDMD uses an orthogonal projection with respect to the distribution~$\mu_0$ of the initial data~$\{x_i\}$, and so approximation~\eqref{eq:T_n_ref} below is equivalent to it. EDMD has been shown in~\cite{KNKWKSN17} to be strongly connected to other established analytic tools for (molecular) dynamical data, such as time-lagged independent component analysis (TICA)~\cite{PHEtAl13,SchPa13}, blind source separation~\cite{MoSch94}, and the variational approach to conformation analysis~\cite{NoNu13}. 
\end{remark}

\section{Time-homogeneous systems and non-stationary data}
\label{sec:nonstat timehomog}


In this final section we illustrate how the above methods can be used to construct a GMSM for and assess properties of a stationary system, even if the simulation data at our disposal does not sample the stationary distribution of the system. In the first example we reconstruct the equilibrium distribution of a reversible system---hence we are able to build an equilibrium GMSM. In the second example we approximate a non-reversible stationary system (i.e., detailed balance does not hold) by a (G)MSM, again from non-stationary data.

Of course, all the examples presented so far can also be computed by the data-based algorithm of section~\ref{sec:data based approx}.

\subsection{Equilibrium MSM from non-equilibrium data}
\label{ssec:eqMSMneqData}

When working with simulation data, we need to take into account that this data might not be in equilibrium. Then, obviously, the empirical distribution does not reflect the stationary distribution of the system. In general, any empirical statistical analysis (e.g., counting transitions between a priori known metastable states) will be biased in such a case.

Let us consider a reversible system with equilibrium distribution~$\mu$, and let the available trajectory data be $\rhoref$-distributed. Then, it is natural to describe the system by its transfer operator $\T_{\text{ref}}:L^2_{\rhoref}\to L^2_{\rhoref}$ with respect to the reference distribution~\cite{KlKoSch16, WuEtAl17}; given explicitly by
\begin{equation} \label{eq:T_ref}
\T_{\text{ref}}\, u(x) = \frac{1}{\rhoref(x)} \int_{\X} u(y) \rhoref(y)\, p^t(y,x)\,\mathrm{d}y\,.
\end{equation}
Note that~$\mu_{\text{corr}} := \mu / \rhoref$ is the stationary distribution of this transfer operator, hence we can retrieve the equilibrium distribution of the system by correcting the reference distribution,~$\mu = \mu_{\text{corr}} \rhoref$.

In the data-based context, we choose the same basis~$\chi_0 = \chi_1$ for initial and final times, since the system is time-homogeneous. In complete analogy to~\eqref{eq:T_N_repres} above, the~$\rhoref$-orthogonal projection of~$\T_{\text{ref}}: L^2_{\rhoref}\to L^2_{\rhoref}$ to~$\set{V}_0$ is given by the matrix
\begin{equation} \label{eq:T_n_ref}
T_{\text{ref},n} = C_{00}^{-1}C_{01}\,.
\end{equation}
We will now apply this procedure to the double-well system from section~\ref{ssec:stationaryDW} with initial points~$x_1,\ldots,x_m$ distributed as shown in Figure~\ref{fig:doublewell_neqData_invdist} (gray histogram). We chose the number of points to be~$m=10^5$, the basis functions $\chi_{0,i}$ to be indicator functions of subintervals of an equipartition of the interval~$[-2,2]$ into $n=100$ subintervals, and the lag time~$\tau=10$. In a preprocessing step we discard all basis functions that do not have any of the points~$x_i$ in their support, thus obtaining a non-singular~$C_{00}$, and use the remaining~$77$ to compute~$T_{n,\text{ref}}\in \R^{77\times 77}$.
\begin{figure}[htb]
\centering

\includegraphics[width = 0.49\textwidth]{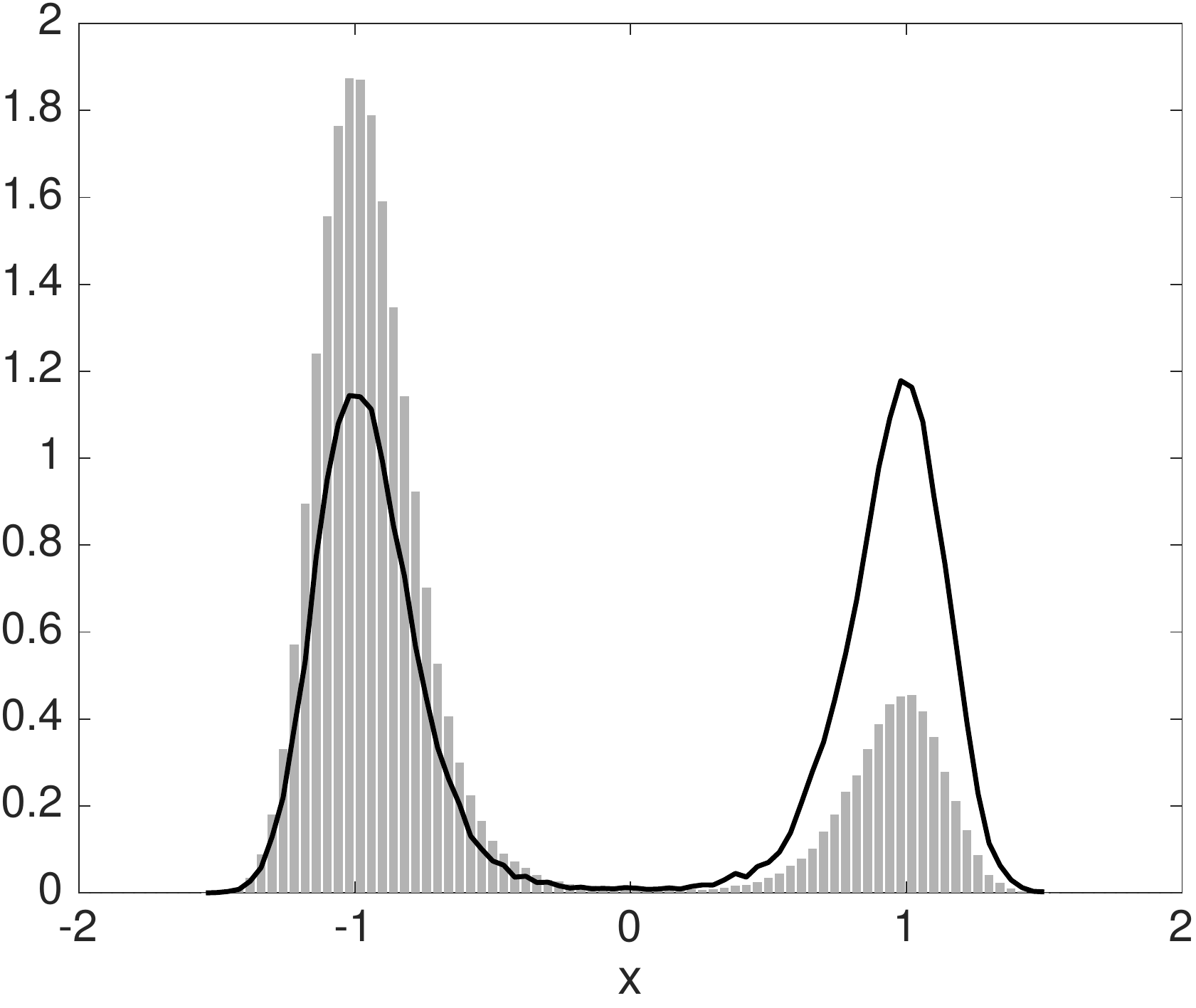}
\hfill
\includegraphics[width = 0.49\textwidth]{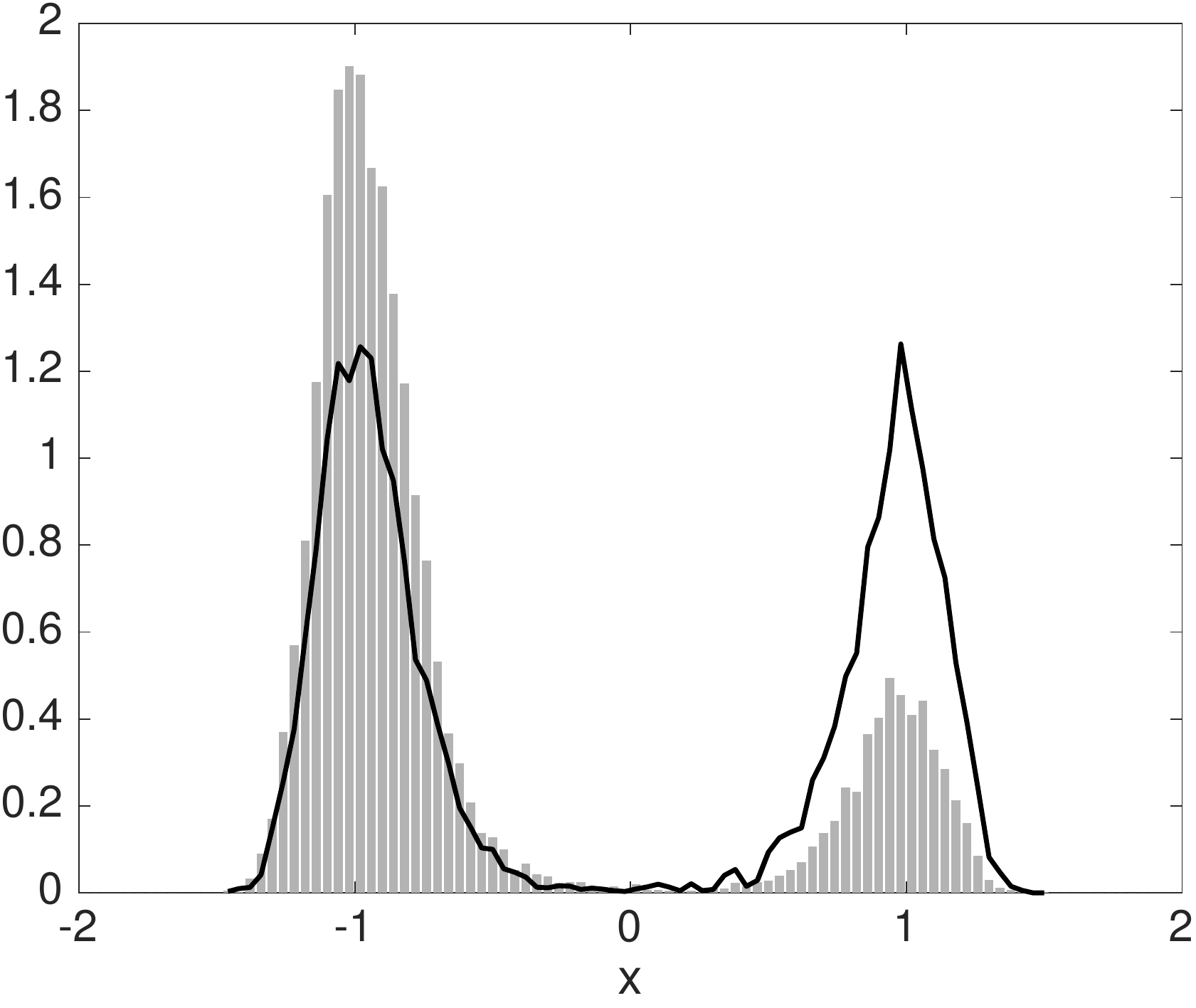}

\caption{The empirical initial distribution of the simulation data, i.e., the reference distribution~$\rhoref$ (gray histogram), and the corrected equilibrium distribution computed from this data (solid black line). Left: sample size $m=10^5$, right: sample size $m=10^4$.}
\label{fig:doublewell_neqData_invdist}
\end{figure}
We obtain $\lambda_2 = 0.894$ giving a time scale $t_2 = 89.6$, and the corrected equilibrium distribution---$\mu = \mu_{\text{corr}} \rhoref$, where $\mu_{\text{corr}}$ is the right eigenvector of $T_{\text{ref},n}$ at eigenvalue 1---is shown in Figure~\ref{fig:doublewell_neqData_invdist} (left) by the black curve. The right-hand side of this figure shows the results of the same computations, but for a sample size~$m=10^4$. Then, we obtain an eigenvalue $0.890$ and corresponding time scale~$85.9$.

It is now simple to reconstruct the approximation~$\T_n$ of~$\T$, the transfer operator with respect to the equilibrium density. Let~$D_{\text{corr}}$ denote the diagonal matrix with the elements of~$\mu_{\text{corr}}$ as diagonal entries. Then,~$T_n = D_{\text{corr}}^{-1}T_{n,\text{ref}} D_{\text{corr}}$ approximates the matrix representation of~$\T_n$ with respect to our basis of step functions.

\begin{remark}[Koopman reweighting] \label{rem:reversibilization}
One can make use of the knowledge that the system that one estimates is reversible, even though due to the finite sample size~$m$ this is not necessarily valid for~$T_{\text{ref},n}$. In~\cite{WuEtAl17}, the authors add for each sample pair~$(x_i,y_i)$ also the pair~$(x_{i+m} = y_i, y_{i+m} = x_i)$ to the sample set, thus numerically forcing the estimate to be reversible. In practice, one defines the diagonal matrix~$\dat{W}$ with diagonal~$\dat{\chi}^T\mu_{\text{corr}}$, builds the reweighted correlation matrices~$\bar{C}_{00} = \frac12 ( \dat{\chi}_0\dat{W}\dat{\chi}_0^T + \dat{\chi}_1\dat{W}\dat{\chi}_1^T )$ and~$\bar{C}_{01} = \frac12 ( \dat{\chi}_1\dat{W}\dat{\chi}_0^T + \dat{\chi}_0\dat{W}\dat{\chi}_1^T )$, and uses them instead of~$C_{00}, C_{01}$.
\end{remark}

\subsection{A non-reversible system with non-stationary data}
\label{eq:nonrev}

Reversible dynamics gives rise to self-adjoint transfer operators, and their theory of Markov state modeling is well developed. However, transfer operators of non-reversible systems are not self-adjoint, hence their spectrum is in general not purely real-valued. Thus, the definition of time scales, and in general the approximation by GMSMs is not fully evolved. Complex eigenvalues indicate cyclic behavior of the process. As this topic is beyond the scope of this paper, we refer the reader to~\cite{DeJu99,DjWeSch16,FrKo17} and to~\cite{DjBaSch15,KnSp15} for Markov state modeling with cycles.

We will consider a non-reversible system here, and show that restricting its behavior to the dominant singular modes of its transfer operator is able to reproduce its dominant long-time behavior, and even allows for a good, few-state MSM. Note that the best rank-$k$ GMSM~\eqref{eq:MSMneq} maps to the $k$-dimensional subspace~$\set{V}_1$ of left singular vectors, thus its eigenvectors also fall into this subspace.

The system in consideration consists of two driving ``forces'', one is a reversible part~$F_r(x) = -\nabla W(x)$ coming from the potential
\[
W(x) = \cos(7\varphi) + 10(r-1)^2,\quad\text{where }x = \begin{pmatrix}
r\cos(\varphi) \\
r\sin(\varphi)
\end{pmatrix},
\]
and the other is a circular driving given by
\[
F_c(x) = \mathrm{e}^{-\beta W(x)} \, \begin{pmatrix}
0 & 1\\ -1 & 0
\end{pmatrix}x\,,
\]
where $\beta=2$ is the inverse temperature, as in~\eqref{eq:overdampedLangevin}. The dynamics now is governed by the SDE~$ \mathrm{d} x_t = (F_r+F_c)(x_t)\,\mathrm{d}t + \sqrt{2\beta^{-1}} \, \mathrm{d}w_t$. It is a diffusion in a 7-well potential (the wells are positioned uniformly on the unit circle) with an additional clockwise driving that is strongest along the unit circle and decreases exponentially in the radial distance from this circle.

For our data-based analysis we simulate a trajectory of this system of length 500 and sample it every 0.01 time instances to obtain an initial set of $5\cdot 10^4$ points. Every point herein is taken as initial condition of 100 independent simulations of the SDE for lag time $\tau=1$, thus obtaining~$5\cdot 10^6$ point pairs~$(x_i,y_i)$.
\begin{figure}[htb]
\centering

\includegraphics[width = 0.32\textwidth]{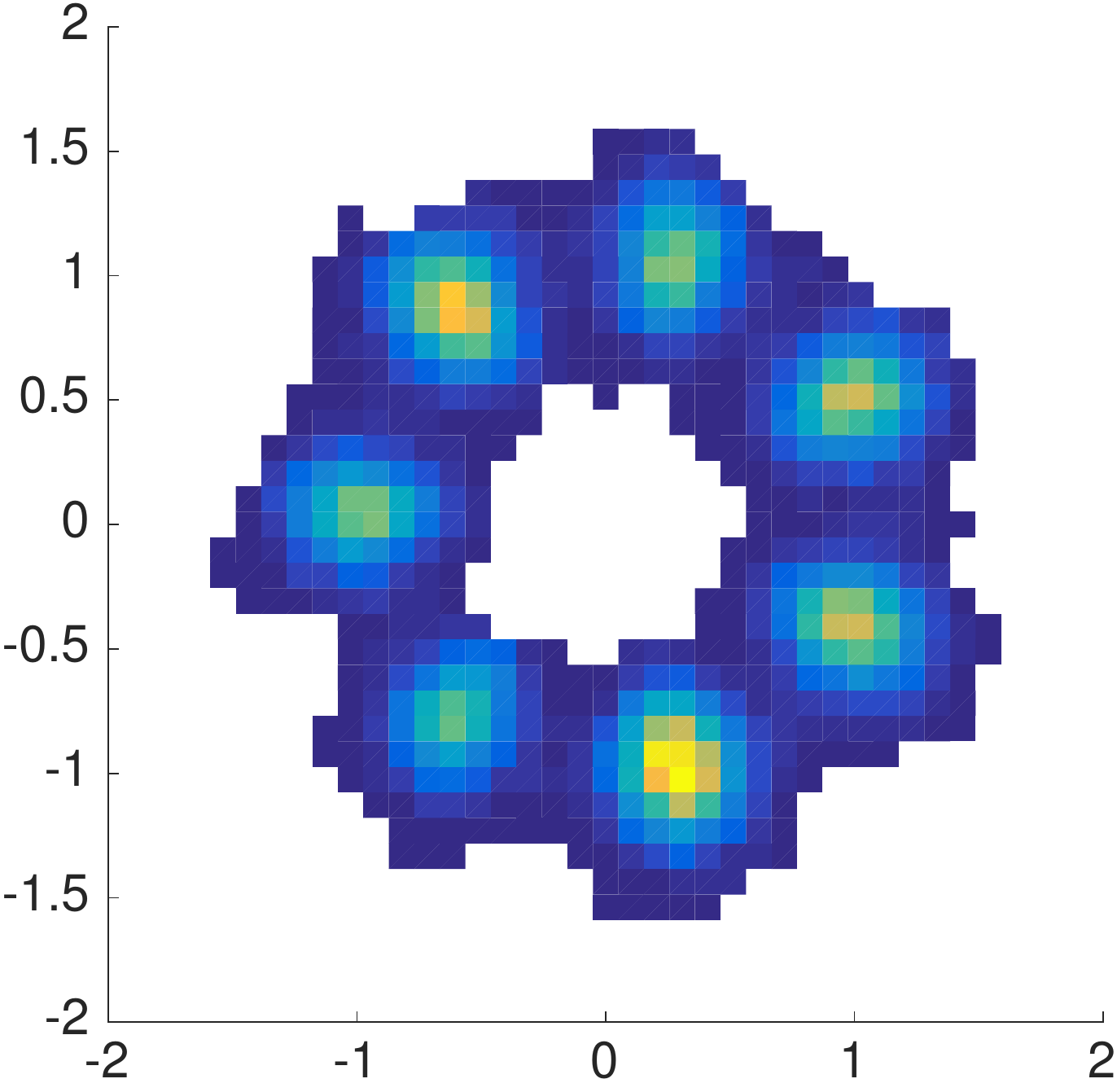}
\hfill
\includegraphics[width = 0.32\textwidth]{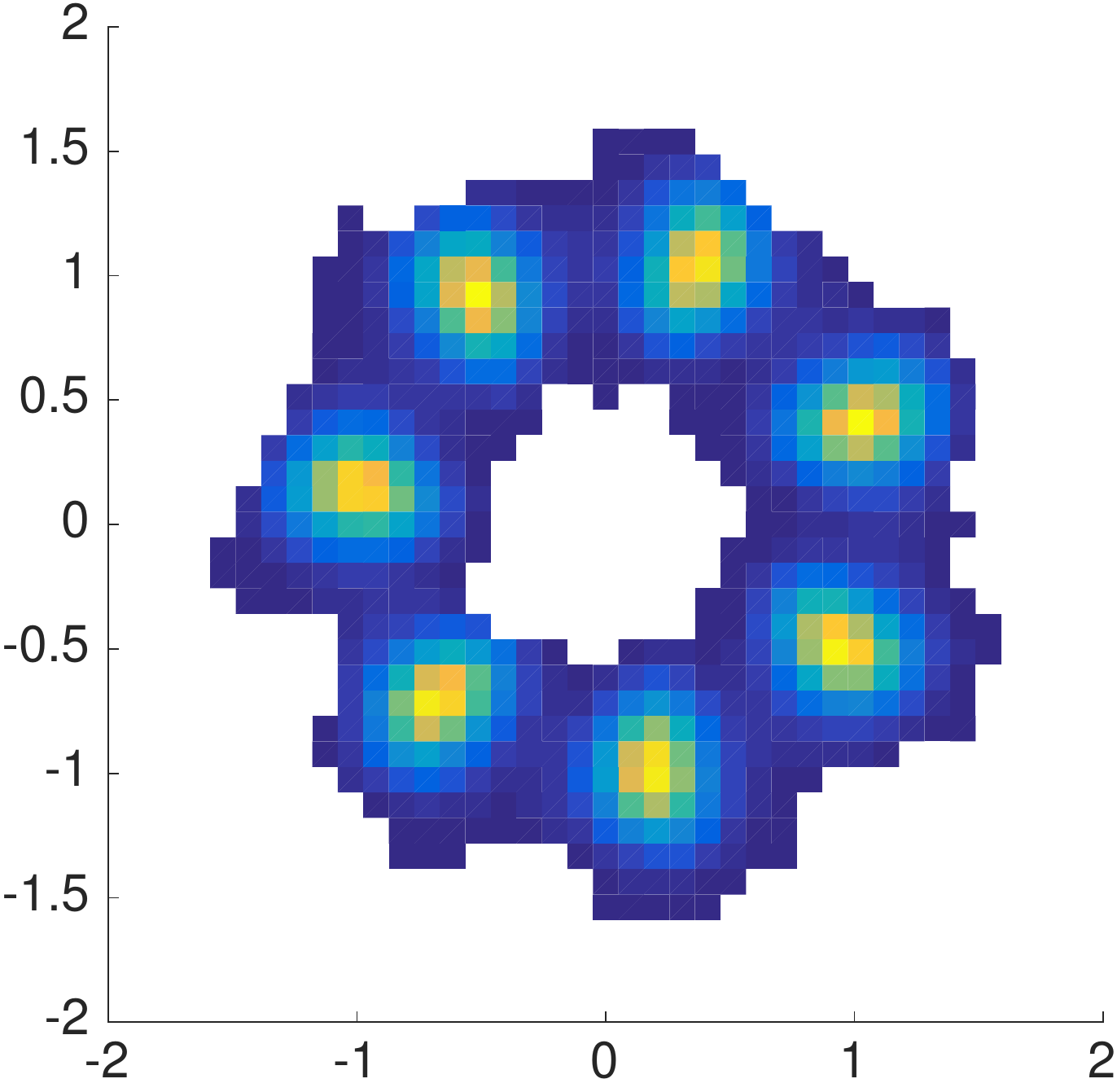}
\hfill
\includegraphics[width = 0.32\textwidth]{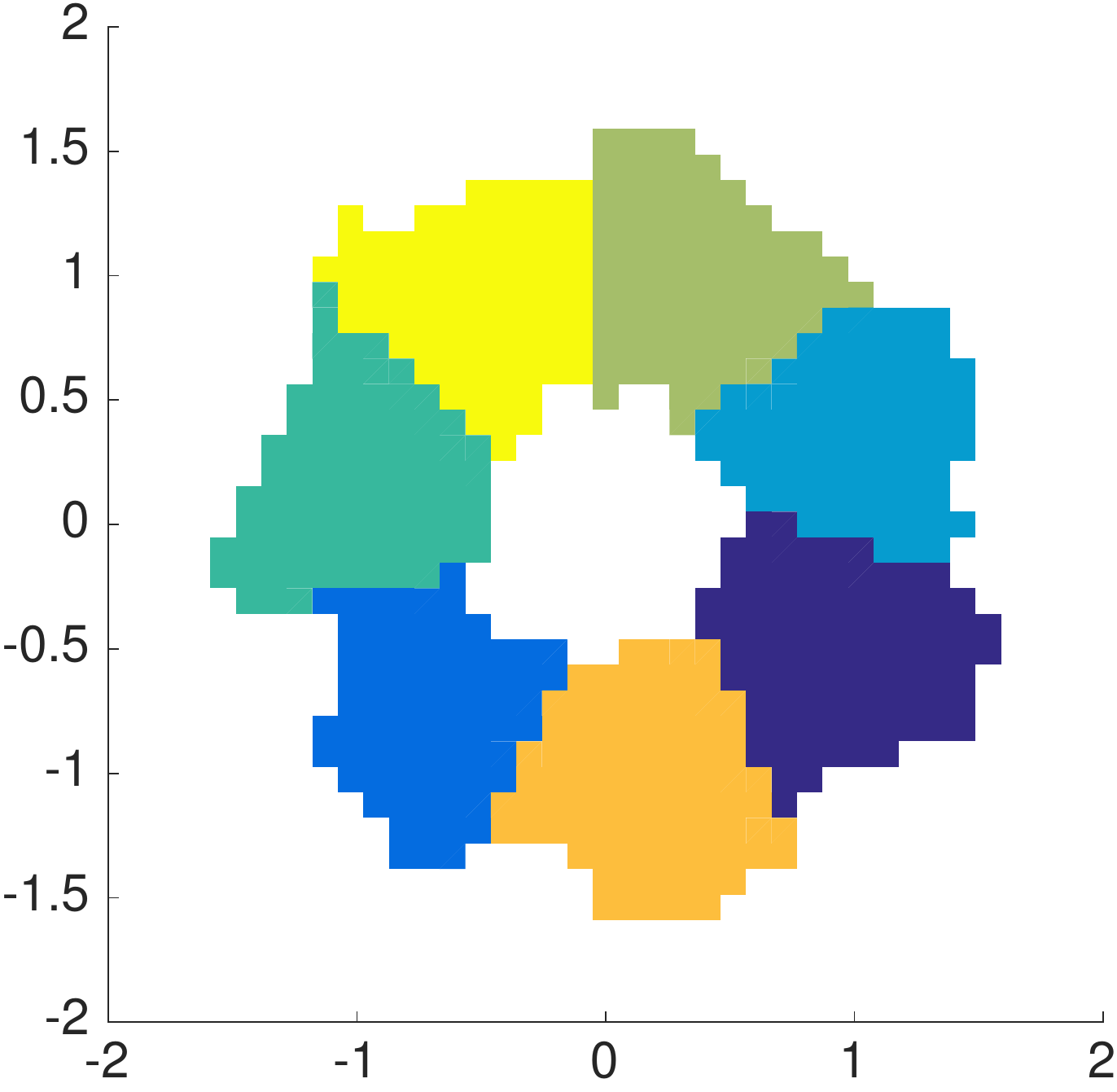}

\caption{Left: empirical distribution of the $x_i$ (histogram with $40\times 40$ bins). Middle: corrected invariant distribution. Right: clustering of the populated partition elements based on the 7 dominant eigenfunctions of the low-rank GMSM.}
\label{fig:lemon_circ_dists}
\end{figure}
We observe in Figure~\ref{fig:lemon_circ_dists} (left) that the empirical distribution of the $x_i$ did not yet converge to the invariant distribution of the system, which would populate every well evenly.

To approximate the transfer operator we use $\chi = \chi_0 = \chi_1$ consisting of the characteristic functions of a uniform $40\times 40$ partition of~$[-2,2]\times [-2,2]$, and restrict this basis set to those~$683$ partition elements that contain at least one~$x_i$ and~$y_j$. The associated projected transfer operator, $T_n$ from~\eqref{eq:T_N_repres} is then used to compute $\tilde{T}_n$ from~\eqref{eq:T_n_newbasis}, and its singular values
\[
\sigma_1 = 1.000,\ 
\sigma_2 = 0.872,\
\sigma_3 = 0.588,\
\ldots,\
\sigma_7 = 0.383,\
\sigma_8 = 0.052,
\]
indicating a gap after seven singular values. Thus, we assemble a rank-7 GMSM~$T_k$ via~\eqref{eq:bestMSM}. This GMSM maps $L^2_{\mu_0}$ to~$L^2_{\mu_1}$, thus to make sense of its eigenmodes, we need to transform its range to densities with respect to~$\mu_0$ instead of~$\mu_1$. As a density~$u$ with respect to~$\mu_1$ is made by~$\frac{\mu_1 u}{\mu_0}$ to a density with respect to~$\mu_0$,
\[
T_k' = C_{00}^{-1}C_{11}T_n
\]
rescales the GMSM to map~$L^2_{\mu_0}$ to itself.\footnote{Note here that since the basis functions are characteristic functions with disjoint support, the correlation matrices $C_{00}, C_{11}$ are diagonal, having exactly the empirical distributions as diagonal entries---i.e., the number of data point falling into the associated partition element.}
We are also interested in the system's invariant distribution. As in section~\ref{ssec:eqMSMneqData}, we can correct the reference distribution~$\rhoref = \mu_0$ by the first eigenfunction~$\mu_{\text{corr}}$ of~$T_k'$ to yield the invariant distribution~$\mu = \mu_{\text{corr}} \rhoref$, cf.~Figure~\ref{fig:lemon_circ_dists} (middle). The dominant eigenvalues of~$ T_k' $ are
\begin{align*}
\lambda_{k,1}' = 0.998 + 0.000i,&\quad \lambda_{k,2/3}' =  0.803 \pm 0.261i, \\
\quad \lambda_{k,4/5}' = 0.511 \pm 0.230i,&\quad \lambda_{k,6/7}' =    0.378 \pm 0.077i\,,
\end{align*}
Note that~$\lambda_{k,1}'<1$. This is due to our restriction of the computation to certain partition elements, as specified above. This set of partition elements is not closed under the process dynamics, and this ``leakage of probability mass'' (about $0.2\%$) is reflected by the dominant eigenvalue. All eigenvalues  of $ T_k' $ are within $0.5\%$ error from the dominant eigenvalues of the transfer operator~$ T_n' $ with respect to the stationary distribution (projected on the same basis set, and computed with higher accuracy), which is a surprisingly good agreement. 

The $8$-th eigenvalue of~$ T_n' $ is smaller in magnitude than~$ 0.03 $.
As indicated by this spectral gap, we may obtain a few-state MSM~$\widehat{T}_k$ here as well. To this end we need to find ``metastable sets'' (although in the case of this cyclically driven system the term metastability is ambiguous) on which we can project the system's behavior. Let~$v_i = (v_{i,1},\ldots,v_{i,n})^T$ denote the $i$-th eigenvector of~$T_k'$. As in the reversible case, where eigenvectors are close to constant on metastable sets, we will seek also here for regions that are characterized by almost constant behavior of the eigenvectors. More precisely, if the $p$-th and~$q$-th partition elements belong to the same metastable set, then we expect~$v_{i,p}\approx v_{i,q}$ for $i=1,\ldots,7$. Thus, we embed the $p$-th partition element into~$\C^7 \equiv \R^{14}$ (i.e., a complex number is represented by two coordinates: its real and imaginary parts) by~$p\mapsto (v_{1,p},\ldots, v_{7,p})^T$, and cluster the hence arising point cloud into $7$ clusters by the $k$-means clustering algorithm.\footnote{The $k$-means algorithm provides a \emph{hard} clustering; i.e., every point belongs entirely to exactly one of the clusters. An automated way to find \emph{fuzzy} metastable sets from a set of eigenvectors is given by the PCCA+ algorithm~\cite{PCCAplus}. A fuzzy clustering assigns to each point a set of non-negative numbers adding up to~1, indicating the \emph{affiliations} of that point to each cluster.}
The result is shown in Figure~\ref{fig:lemon_circ_dists} (right). Taking these sets we can assemble the MSM~$\widehat{T}_k \in\R^{7\times 7}$ via~\eqref{eq:met_transrates}. We obtain a MSM that maps a Markov state (i.e., a cluster) with probability $0.62$ to itself, with probability $0.29$ to the clockwise next cluster, and with probability~$0.06$ to the second next cluster in clockwise direction. The probability to jump one cluster in the counterclockwise direction is below~$0.001$. The eigenvalues of~$\widehat{T}_k$,
\begin{align*}
\widehat{\lambda}_1 = 0.998 + 0.000i,&\quad \widehat{\lambda}_{2/3} =  0.800 \pm 0.260i, \\
\quad \widehat{\lambda}_{4/5} = 0.507 \pm 0.227i,&\quad \widehat{\lambda}_{6/7} =    0.374 + 0.076i\,,
\end{align*}
are also close to those of~$ T_n' $ (below $1 \%$ error), justifying this MSM.

\section*{Acknowledgments}

This work is supported by the Deutsche Forschungsgemeinschaft (DFG) through the CRC 1114 ``Scaling Cascades in Complex Systems'', projects A04 and B03, and the Einstein Foundation Berlin (Einstein Center ECMath).

\begin{appendix}

\section{Optimal low-rank approximation of compact operators}
\label{app:E-Ythm}

For completeness, we include a proof of the Eckart--Young--Mirsky theorem for compact operators between separable\footnote{A space is separable if it has a countable basis. The Lebesgue space~$L^2_{\mu}(\R^d)$ of $\mu$-weighted square-integrable functions is separable for bounded and integrable~$\mu$. This is the case we consider here.}
Hilbert spaces. In particular, it shows that the optimal low-rank approximation of such an operator is obtained by an \emph{orthogonal} projection on its subspace of dominant singular vectors; cf.~\eqref{eq:optlowrank}.

\begin{theorem}
Let $\A: \set{H}_0 \to \set{H}_1$ be a compact linear operator between the separable Hilbert spaces $\set{H}_0$ and $\set{H}_1$, with inner products $\innerprod{\cdot}{\cdot}_0$ and~$\innerprod{\cdot}{\cdot}_1$, respectively. Then, the optimal rank-$k$ approximation $\A_k$ of $\A$ in the sense that
\[
\| \A - \A_k \| \to \min_{\text{rank}\, \A_k = k}\,,
\]
where $\| \cdot \|$ denotes the induced operator norm, is given by
\begin{equation} \label{eq:optlowrank}
\A_k = \sum_{i=1}^k \sigma_i \psi_i \innerprod{\phi_i}{\cdot}_0\,,
\end{equation}
where $\sigma_i,\psi_i,\phi_i$ are the singular values (in non-increasing order), left and right normalized singular vectors of~$\A$, respectively. The optimum is unique iff $\sigma_k > \sigma_{k+1}$.
\end{theorem}
\begin{proof}
Let $\A_k$ be defined as in~\eqref{eq:optlowrank}. Since $\A = \sum_{i=1}^{\infty} \sigma_i \psi_i \innerprod{\phi_i}{\cdot}_0$, we have
\begin{equation} \label{eq:optestim}
\| \A - \A_k \| = \| \sum_{i=k+1}^{\infty} \sigma_i \psi_i \innerprod{\phi_i}{\cdot}_0 \| = \sigma_{k+1}\,.
\end{equation}
Let now $\op{B}_k$ be any rank-$k$ operator from $\set{H}_0$ to $\set{H}_1$. Then, there exist linear functionals $c_i:\set{H}_0\to \R$ and vectors~$v_i\in\set{H}_1$, $i=1,\ldots,k$, such that
\[
\op{B}_k = \sum_{i=1}^k c_i(\cdot) v_i\,.
\]
For every~$i$, since $c_i$ has one-dimensional range, its kernel has co-dimension~1, thus the intersection of the kernels of all the~$c_i$ has co-dimension at most~$k$. Thus, any $(k+1)$-dimensional space has a non-zero element~$w$ with~$c_i(w) = 0$ for $i=1,\ldots,k$.

By this, we can find scalars~$\gamma_1,\ldots,\gamma_{k+1}$ such that~$\sum_{i=1}^{k+1}\gamma_i^2 = 1$ and~$w = \gamma_1\phi_1 + \ldots + \gamma_{k+1}\phi_{k+1}$ satisfies~$c_i(w) = 0$ for $i=1,\ldots,k$. By construction $\|w\|_0$ holds. It follows that
\[
\| \A - \op{B}_k \|^2 \ge \| (\A - \op{B}_k)w \|_1^2 = \| \A w \|_1^2 = \sum_{i=1}^{k+1} \sigma_i^2\gamma_i^2 \ge \sigma_{k+1}^2 \underbrace{\sum_{i=1}^{k+1} \gamma_i^2}_{=1}\,.
\]
This with~\eqref{eq:optestim} proves the claim.
\end{proof}
As a corollary, if $\A:\set{H} \to \set{H}$ is a self-adjoint operator, then its eigenvalue and singular value decompositions coincide, giving~$\psi_i = \phi_i$, and thus~$\A_k$ in~\eqref{eq:optlowrank} is the projection on the dominant eigenmodes.
\end{appendix}

\small
\bibliographystyle{alpha}
\bibliography{References}
\end{document}